\documentclass{amsart}
\usepackage{doi}
\usepackage{amsmath, amssymb, mathrsfs, bbold, hyperref,enumerate, stix, stmaryrd}
\SetSymbolFont{stmry}{bold}{U}{stmry}{m}{n}
\usepackage{graphicx,wrapfig}
\graphicspath{{pic/}}
\usepackage[utf8]{inputenc}
\usepackage[T1]{fontenc}
\usepackage[english]{babel}
\usepackage[hmargin=3cm,vmargin=3cm]{geometry}
\usepackage{color}
\usepackage[open, openlevel=2, depth=3, atend]{bookmark}
\hypersetup{pdfstartview=XYZ}
\usepackage{anyfontsize}
\usepackage{enumitem}
\usepackage{tikz}
\usetikzlibrary{positioning}

\makeatletter
\numberwithin{equation}{section}
\numberwithin{figure}{section}

 \theoremstyle{plain}
 \newtheorem{thm}{Theorem}
 \theoremstyle{plain}
 \newtheorem*{thm*}{Theorem}
  \theoremstyle{plain}
  \numberwithin{thm}{section}
  
  \theoremstyle{plain}
 \newtheorem*{cor*}{Corollary}
  \theoremstyle{plain}
  \newtheorem{lem}[thm]{Lemma}
  \theoremstyle{remark}
  \newtheorem{rem}[thm]{Remark}
    \theoremstyle{remark}
    
    \theoremstyle{remark}

   \theoremstyle{plain}
  \newtheorem{prop}[thm]{Proposition}
     \theoremstyle{plain}
  \newtheorem{definition}[thm]{Definition}
   \theoremstyle{plain}
  
   \theoremstyle{plain}
     
   \theoremstyle{plain}

  \def\Ddots{\mathinner{\mkern1mu\raise\p@
\vbox{\kern7\p@\hbox{.}}\mkern2mu
\raise4\p@\hbox{.}\mkern2mu\raise7\p@\hbox{.}\mkern1mu}}
\makeatother

\newcommand{\norm}[1]{\left\| #1 \right\|}

\newcommand{\eklm}[1]{\left\langle #1 \right\rangle}

\newcommand{\N}{{\mathbb N}}

\newcommand{\C}{{\mathbb C}}

\newcommand{\R}{{\mathbb R}}

\newcommand{\B}{{\mathcal B}}
\newcommand{\D}{{\mathcal D}}

\newcommand{\M}{{\mathcal M}}

\newcommand{\X}{{\mathfrak X}}
\newcommand{\Xbf}{{\mathbf X}}

\newcommand{\V}{{\mathcal V}}

\renewcommand{\epsilon}{\varepsilon}

\renewcommand{\rho}{\varrho}

\newcommand{\bdm}{\begin{displaymath}}
\newcommand{\edm}{\end{displaymath}}
\newcommand{\bq}{\begin{equation}}
\newcommand{\eq}{\end{equation}}
\newcommand{\bqn}{\begin{equation*}}
\newcommand{\eqn}{\end{equation*}}

\newcommand{\Cinft}{{\rm C^{\infty}}}

\renewcommand{\L}{{\rm L}}

\newcommand{\SO}{\mathrm{SO}}

\newcommand{\g}{{\bf \mathfrak g}}

\newcommand{\aL}{{\bf \mathfrak a}}
\newcommand{\nL}{{\bf \mathfrak n}}
\renewcommand{\k}{{\bf \mathfrak k}}
\newcommand{\m}{{\bf \mathfrak m}}

\newcommand{\p}{{\bf \mathfrak p}}

\newcommand{\so}{{\bf \mathfrak so}}

\newcommand{\eps}{\varepsilon}

\newcommand{\Ad}{\mathrm{Ad}}
\newcommand{\ad}{\mathrm{ad}}

\renewcommand{\Im}{\mathrm{Im}\,}
\renewcommand{\Re}{\mathrm{Re}\,}

\newcommand{\pr}{\mathrm{pr}}

\newcommand{\gam}{\Gamma\backslash}

\DeclareMathOperator{\Res}{Res}

\newcommand{\hooklongrightarrow}{\lhook\joinrel\longrightarrow}
\definecolor{darkblue}{rgb}{0,0,0.4}

\newcommand{\tu}[1]{\textup{#1}}

\begin{document}
\title{Pollicott-Ruelle resonant states and betti numbers}
\author{Benjamin K\"uster, Tobias Weich}
\email{benjamin.kuster@math.u-psud.fr, weich@math.uni-paderborn.de}
\maketitle
\begin{abstract}
Given a closed orientable  hyperbolic manifold of dimension $\neq 3$ we prove that the multiplicity of the Pollicott-Ruelle resonance of the geodesic flow on perpendicular one-forms at zero agrees with the first Betti number of the manifold. Additionally,  we prove that this equality is stable under small perturbations of the Riemannian metric and simultaneous small perturbations of the geodesic vector field within the class of contact vector fields. For more general perturbations we get bounds on the multiplicity of the resonance zero on all one-forms in terms of the first and zeroth Betti numbers.  Furthermore, we identify for hyperbolic manifolds further resonance spaces whose multiplicities are given by higher Betti numbers. 
\end{abstract}

\section*{Introduction}
Pollicott-Ruelle resonances have been introduced in the 1980's in order to study mixing properties of hyperbolic flows and can nowadays be understood as a discrete spectrum of the generating vector field (see Section~\ref{subsec:rp_res} for a definition and references). Very recently it has been discovered that in certain cases some particular Pollicott-Ruelle resonances have a topological meaning. Let us recall these results:

In \cite{dyatlov-zworski17} Dyatlov and Zworski prove that on a closed orientable surface $\mathcal M$ of negative curvature the Ruelle zeta function at zero vanishes to the order $|\chi(\mathcal M)|$, where $\chi(\mathcal M)$ is  the Euler characteristic of $\M$, generalizing a result of Fried in constant curvature \cite{fried86}\footnote{The methods of Fried to study the Ruelle zeta function generalize to locally symmetric spaces culminating in the recent work of Shen \cite{She18}.}. Dyatlov and Zworski prove their result as follows: 
By previous results on the meromorphic continuation of the Ruelle zeta function (see  \cite{DZ16,GLP13}) the order of vanishing of the Ruelle zeta function at zero can be  expressed as the alternating sum $\sum_{k=0}^2 (-1)^ {k+1} m_{\mathcal L_X, \Lambda^k X^\perp}(0)$, where $m_{\mathcal L_X, \Lambda^k X^\perp }(0)$ is the  multiplicity of the resonance zero of the Lie derivative $\mathcal L_X$ along the geodesic vector field $X\in \Gamma^\infty(T(S^\ast \M))$ acting on perpendicular $k$-forms. The latter are those $k$-forms on the unit co-sphere bundle $S^\ast \M$ that vanish upon contraction with $X$ (for the precise definition of the multiplicities, see Sections~\ref{sec:flows} and \ref{subsec:rp_res}).  
For closed orientable surfaces it is rather easy to see that 
$m_{\mathcal L_X, \Lambda^0 X^\perp}(0) = m_{\mathcal L_X, \Lambda^2 X^\perp}(0)= b_0(\mathcal M)=b_2(\mathcal M)$, thus the central task is to prove that $m_{\mathcal L_X, X^\perp}(0)= b_1(\mathcal M)$. 
Dyatlov and Zworski achieve this by combining  microlocal analysis with Hodge theory \cite[Proposition 3.1(2)]{dyatlov-zworski17}. 
This is a remarkable result  also apart from its implications on zeta function questions   because it identifies a resonance whose multiplicity has a precise topological meaning. 

Let us mention a second result establishing a connection between Pollicott-Ruelle resonances and topology: 
Dang and Rivi\`ere \cite{DR17b} examine a general Anosov flow $\varphi_t= e^{Yt}$ on a closed orientable   manifold. The Lie derivative $\mathcal L_Y$ has a discrete spectrum (the Pollicott-Ruelle spectrum) on certain spaces of anisotropic $p$-currents and it is shown that the exterior derivative acting on generalized eigenspaces of the eigenvalue zero forms a complex which is quasi-isomorphic to the de Rham complex.\footnote{We would like to point out that an analogous  statement also holds for Morse-Smale flows \cite{DR16, DR17, DR17a, DR17b} and in these  cases the spectral complex defined by the Pollicott-Ruelle resonances is actually isomorphic to the Morse complex. Consequently, the spectral complex of Dang and Rivi\`ere  can be considered as a generalization of the Morse complex to Anosov flows.}
While this result gives no precise information about the multiplicities of the resonances, it gives lower bounds for them and it holds in very great generality. 

As a third result we would like to mention \cite{GHW18} where the relation between Pollicott-Ruelle and quantum resonances is studied for compact and convex co-compact hyperbolic surfaces. 
For this correspondence the resonances at negative integers turn out to be exceptional points and it is shown that their multiplicities can be expressed by the Euler characteristic of the hyperbolic surface. 
The proof uses a Poisson transform to establish a bijection between the resonant states and holomorphic sections of certain line bundles, and the formula for the multiplicities follows from a Riemann-Roch theorem. 

In the present article we broaden the picture regarding  the topological properties of Pollicott-Ruelle resonant states. 
To this end, we combine some of the above approaches:
In a first step we use a quantum-classical correspondence to find new examples of resonances with topological multiplicities. 
In particular, we prove
\begin{prop}\label{prop:intro_perp}
For any closed orientable hyperbolic manifold $\mathcal M$ of dimension $n+1$ with $n\neq 2$, one has 
 \[
 m_{\mathcal L_X, X^\perp}(0)= b_1(\mathcal M).
 \]
 Furthermore, the resonance zero has no Jordan block and if $n\geq 3$, then zero is the unique leading resonance  and there is a spectral gap.\footnote{See the paragraph below \eqref{eq:WF_characterization} for the definition of ``having no Jordan block'' and the footnote in Prop.\ \ref{prop:constant_mult_perp} for the other terms used here.}
\end{prop}
We prove these statements using the general framework of vector-valued quantum-classical correspondence developed by the authors \cite{kuester-weich18} as well as a Poisson transform of Gaillard \cite{gaillard1}.\footnote{It has been noted in \cite[Remark 5]{DGRS18} (without detailing the proof) that the statement of Proposition~\ref{prop:intro_perp} can alternatively be obtained by a zeta factorization argument similar to \cite[Proposition 7.7]{DGRS18} based on the work of Bunke and Olbrich \cite{BO95}.} Without any further effort these ingredients provide additional examples of resonance multiplicities related to not only the first but to all Betti numbers, see Proposition~\ref{prop:p_forms_Betti}. More precisely, the latter result shows that the $p$-th Betti number of a closed orientable  hyperbolic manifold can be recovered as the dimension of the space of some particular resonant $p$-forms in the kernel of a so-called  \emph{horocycle operator} (see Section \ref{sec:horocycle}).  
For $n=1$ the first statement in Proposition~\ref{prop:intro_perp} is the special case of \cite[Proposition 3.1(2)]{dyatlov-zworski17} restricted to hyperbolic surfaces. Interestingly $n=2$ is an exceptional case and the multiplicity is given by $m_{\mathcal L_X, X^\perp}(0)= 2b_1(\mathcal M)$ (see Remark~\ref{rem:exceptional}). For $n> 2$ the statement can be considered as a generalization of the Dyatlov-Zworski result to higher dimensions at the cost of restricting to manifolds of constant negative curvature. 

In a second step we can partially overcome this restriction and prove
\begin{prop}\label{prop:main2}
 Let $(\mathcal M,\mathscr g_0)$ be a closed orientable  hyperbolic manifold of dimension $n+1$ with $n\neq 2$ and let $\Gamma^{\infty}(\mathrm{S}^2(T^*\mathcal M))$ be the space of smooth symmetric two-tensors endowed with its Fr\'echet topology and $\mathscr R_{\mathcal M,<0}$ the open subset of Riemannian metrics of negative sectional curvature. Then there is an open neighborhood $U\subset\mathscr R_{\mathcal M,<0}$ of $\mathscr g_0$ such that for all Riemannian metrics $\mathscr g\in U$ one has
 \begin{equation}\label{eq:betti_perp}
 m_{\mathcal L_{X_\mathscr g}, X_{\mathscr g}^\perp}(0)= b_1(\mathcal M).
 \end{equation}
 Here $X_{\mathscr g}\in \Gamma^\infty(T(S_{\mathscr g}^\ast \M))$ is the geodesic vector field on the unit co-sphere bundle $S_{\mathscr g}^\ast \M$ with respect to $\mathscr g$. 
\end{prop}
Note that also in dimension $n+1=2$ we obtain the equality \eqref{eq:betti_perp} only in a neighborhood of $\mathscr g_0$, whereas Dyatlov and Zworski prove the equality in this dimension for all $\mathscr g\in\mathscr R_{\M,<0}$. It seems thus reasonable to conjecture that the equality holds in all dimensions $n+1\neq 3$ for all $\mathscr g\in \mathscr R_{\M,<0}$, or at least for all $\mathscr g$ in those connected components of $\mathscr R_{\mathcal M,<0}$ that contain a metric of constant negative curvature. 

We obtain Proposition \ref{prop:main2} as a special case of a more general result on simultaneous perturbations of the Riemannian metric and the geodesic vector field. To state this result, consider in the situation of Proposition \ref{prop:main2} some Riemannian metric $\mathscr g\in\mathscr R_{\mathcal M,<0}$ and an arbitrary Anosov vector field $Y_{\mathscr g}$ on $S_{\mathscr g}^\ast \M$. Denoting by $Y_{\mathscr g}^\perp\subset T^\ast(S_{\mathscr g}^\ast \M)$ the ``perpendicular'' subbundle formed by all co-vectors annihilating $Y_{\mathscr g}$ fiber-wise, the multiplicities of the resonance zero of the Lie derivative $\mathcal L_{Y_{\mathscr g}}$ acting on sections of $Y_{\mathscr g}^\perp$ and $T^\ast(S_{\mathscr g}^\ast \M)$, respectively, are easy to relate under relatively mild assumptions: by Lemma \ref{lem:perpall} it suffices to assume that there is a one-form $\alpha_{\mathscr g}$ on $S_{\mathscr g}^*\mathcal M$ with $\iota_{Y_{\mathscr g}}\alpha_{\mathscr g}=1$, $\iota_{Y_{\mathscr g}}d\alpha_{\mathscr g}=0$, and $T^\ast(S^\ast_{\mathscr g}\M)=\R\alpha_{\mathscr g} \oplus Y_{\mathscr g}^\perp$ to have the relation
\bq
m_{\mathcal L_{Y_{\mathscr g}},T^*(S_{\mathscr g}^*\mathcal M)}(0)=m_{\mathcal L_{Y_{\mathscr g}}, Y_{\mathscr g}^\perp}(0)+ b_0(\mathcal M).\label{eq:decom3592306}
\eq
This is fulfilled, for example, if $\alpha_{\mathscr g}$ is  a contact form and $Y_{\mathscr g}$ is a contact Anosov vector field with respect to $\alpha_{\mathscr g}$. In particular, if $Y_{\mathscr g}=X_{\mathscr g}$ is the geodesic vector field, one can take $\alpha_{\mathscr g}$ to be the canonical contact form given by the restriction of the Liouville one-form to $S^\ast_{\mathscr g}\M$. So  \eqref{eq:betti_perp} is in fact equivalent to the equation
\begin{equation}\label{eq:betti_nonperp}
 m_{\mathcal L_{X_\mathscr g}, T^*(S_{\mathscr g}^*\mathcal M)}(0)= b_0(\mathcal M)+b_1(\mathcal M).
 \end{equation}
In Section \ref{sec:nonconst} we study the stability of the equation \eqref{eq:betti_nonperp} upon simultaneous perturbations of the Riemannian metric and the geodesic vector field. We obtain the following main result:
\begin{thm}\label{thm:main2}
If $\dim \M \neq 3$ and $\mathscr g_0\in \mathscr R_{\mathcal M,<0}$ is a metric of constant negative curvature,  then there exists an open set $U\subset \mathscr R_{\mathcal M,<0}$ containing $\mathscr g_0$ and a constant $\delta>0$ such that for all Riemannian metrics $\mathscr g\in U$ and all vector fields $Y_{\mathscr g}\in \Gamma^\infty(T(S_{\mathscr g}^*\mathcal M))$ with\footnote{See \eqref{eq:C1norm} for the definition of the $\mathrm{C}^1$-norm used here. For small $\Vert Y_{\mathscr g}-X_{\mathscr g}\Vert_{\mathrm{C}^1}$ the vector field $Y_{\mathscr g}$ is Anosov by the structural stability of the Anosov property \cite{kato73}, so that the Pollicott-Ruelle resonances of $\mathcal L_{Y_{\mathscr g}}$ are well-defined.}  $\Vert Y_{\mathscr g}-X_{\mathscr g}\Vert_{\mathrm{C}^1}<\delta$ one has the bounds
\bqn
b_1(\mathcal M) \;\leq \;m_{\mathcal L_{Y_{\mathscr g}},T^*(S_{\mathscr g}^*\mathcal M)}(0)\;\leq\; b_0(\mathcal M) + b_1(\mathcal M),
\eqn
and if there is a one-form $\alpha_{\mathscr g}$ on $S_{\mathscr g}^*\mathcal M$ with  $\mathcal L_{Y_{\mathscr g}}\alpha_{\mathscr g}=0$ and $d\alpha_{\mathscr g}\neq 0$, then the bounds improve to the equality
\bqn
 m_{\mathcal L_{Y_{\mathscr g}},T^*(S_{\mathscr g}^*\mathcal M)}(0)=b_0(\mathcal M) + b_1(\mathcal M).
\eqn
\end{thm}
\begin{rem}
If $\alpha_{\mathscr g}$ is a contact form and $Y_{\mathscr g}$ is contact with respect to $\alpha_{\mathscr g}$, then, as mentioned above, the resonance multiplicities on the bundles $T^*(S_{\mathscr g}^*\mathcal M)$ and $Y_{\mathscr g}^\perp$ are related by \eqref{eq:decom3592306}. So  Theorem \ref{thm:main2} implies that the relations  \eqref{eq:betti_perp} and \eqref{eq:betti_nonperp} remain valid for simultaneous small perturbations $\mathscr g$ of the metric $\mathscr g_0$ and 
small perturbations $Y_{\mathscr g}$ of the geodesic vector field $X_{\mathscr g}$ within the class of contact vector fields. 
\end{rem}
We prove Theorem \ref{thm:main2} by combining Proposition~\ref{prop:intro_perp}, which has been obtained by a quantum-classical correspondence, with the cohomology results of Dang-Rivi\`ere \cite{DR17b} as well as some recent advances concerning the perturbation theory of Pollicott-Ruelle resonances \cite{Bon18}. 

 The main steps in the proof of Proposition  \ref{prop:intro_perp}, carried out in Section \ref{sec:constant}, can be roughly summarized as follows:\begin{enumerate}
\item First we prove that $m_{\mathcal L_X, X^\perp}(0)=m_{\mathcal L_X, E_+^\ast}(0)$, i.e., every generalized resonant state $u$ of the resonance zero actually lives only in the dual stable subbundle  $E_+^\ast\subset X^\perp$.
\item Then we show that $u$ lies in the kernel of the horocyclic operator $\mathcal U_-$ (defined in Section \ref{sec:horocycle}), which means that it is a generalized \emph{first band} resonant state. This is achieved by observing that $\mathcal U_- u$ is a generalized resonant state on the tensor bundle $E_+^\ast\otimes E_-^\ast\cong E_-^\ast\otimes E_-^\ast$. Decomposing   $\mathcal U_- u$ into a symmetric and an antisymmetric part, we apply \cite{dfg} to show that the symmetric part must be zero and \cite{gaillard1} to show that the antisymmetric part must be zero.
\item By \cite{kuester-weich18} there are no first band Jordan blocks, so it follows that $u$ is actually a resonant state. 
\item Since $u$ is a first band resonant state, $u$ corresponds to a distributional one-form $u_\infty$ on the sphere $S^n$, the boundary at infinity of the hyperbolic space $\mathbb H^{n+1}$. Then $u_\infty$ is invariant under a certain representation of the lattice $\Gamma\subset \SO(n+1,1)_0$ on the space of distributional one-forms on $S^n$, where $\M=\gam \mathbb H^{n+1}$. 
\item We apply again  Gaillard's result \cite{gaillard1}; it says that $u_\infty$ is mapped by a Poisson transform to a harmonic one-form on $\M$ which is non-zero if $u$ is non-zero and that all  harmonic one-forms on $\M$ arise this way.
\end{enumerate}
In Section \ref{sec:nonconst} we then carry out the proof of Theorem \ref{thm:main2} along roughly the following steps:
\begin{enumerate}
\item Using the ``fiber-wise rescaling'' diffeomorphism between the unit co-sphere bundles $S^\ast_{\mathscr g}\M$, $S^\ast_{\mathscr g_0}\M$ with respect to two Riemannian metrics ${\mathscr g}$, ${\mathscr g}_0$ on $\M$, we transfer the initial setup involving vector fields $Y_{\mathscr g}\in\Gamma^\infty(T(S^\ast_{\mathscr g}\M))$ to an equivalent  setup involving vector fields $Y_{\ast\mathscr g}\in\Gamma^\infty(T(S^\ast_{\mathscr g_0}\M))$ on the $\mathscr g$-independent space $S^\ast_{\mathscr g_0}\M$. This transfer is such that if $\mathscr g$ is close to $\mathscr g_0$ and $Y_{\mathscr g}$ is close to the geodesic vector field $X_{\mathscr g}$, then $Y_{\ast\mathscr g}$ is close to $X_{\mathscr g_0}$. We choose $\mathscr g_0$  of constant negative curvature.

\item By applying Bonthonneau's result \cite{Bon18} on perturbations of Anosov vector fields to the transferred setup on $S^\ast_{\mathscr g_0}\M$, we obtain the inequality $m_{\mathcal L_{Y_{\mathscr g}},T^*(S_{\mathscr g}^*\mathcal M)}(0)\leq m_{\mathcal L_{X_{\mathscr g_0}},T^*(S_{\mathscr g_0}^*\mathcal M)}(0)$ for all $\mathscr g$ close enough to $\mathscr g_0$ and all vector fields $Y_{\mathscr g}$ close enough to $X_{\mathscr g}$ (they are then automatically Anosov).

\item From the results of Dang-Rivi\`ere \cite{DR17b} we get the lower bound $b_1(\M)\leq m_{\mathcal L_{Y_{\mathscr g}},T^*(S_{\mathscr g}^*\mathcal M)}(0)$ for every negatively curved Riemannian metric $\mathscr g$ on $\M$ and every Anosov vector field $Y_{\mathscr g}$ on $S^\ast_{\mathscr g}\M$, and this bound improves to $b_0(\M)+b_1(\M)$ if $Y_{\mathscr g}$ preserves a non-closed one-form.

\item  In the proof of Proposition  \ref{prop:intro_perp} we observe $m_{\mathcal L_{X_{\mathscr g_0}},T^*(S_{\mathscr g_0}^*\mathcal M)}(0)=b_0(\M)+b_1(\M)$ if $\dim \M\neq 3$.
\end{enumerate}

\subsubsection*{Acknowledgements}
After the appearance of \cite{kuester-weich18}, Semyon Dyatlov raised the question whether the vector-valued quantum-classical correspondence might shed light on generalizations of \cite{dyatlov-zworski17} to higher dimensions. We are grateful to him for proposing this question as well as for several helpful discussions. Furthermore, we thank Colin Guillarmou, Viet Nguyen Dang, and Gabriel Rivi\`ere for several discussions concerning their works \cite{DR17b, DGRS18} and helpful comments. We  thank the anonymous referees for their valuable remarks and, in particular, for pointing out that Proposition \ref{prop:main2} could be strengthened, which led us to prove Theorem \ref{thm:main2}. This project has received funding from the European Research Council (ERC) under the European Union's Horizon 2020 research and innovation programme (grant agreement No.\ 725967)
as well as from the Deutsche Forschungsgemeinschaft (DFG) through the Emmy Noether group ``Microlocal Methods for Hyperbolic Dynamics'' (grant No.\ WE 6173/1-1). 
\section{Pollicott-Ruelle resonances for geodesic flows}
\subsection{Anosov vector fields and perpendicular forms}\label{sec:flows}
Let $(\mathcal M, \mathscr g)$ be a closed orientable  Riemannian manifold of dimension $n+1$ with negative sectional curvature. 
Then the geodesic flow $\varphi_t$ on the unit co-sphere bundle $S^*\mathcal M$ is an Anosov flow which implies that there is a $d\varphi_t$-invariant H\"older continuous splitting of the tangent bundle $T(S^*\mathcal M)$
\bq
 T(S^*\mathcal M) = E_0\oplus E_+\oplus E_-,\label{eq:Anosovsplitting}
\eq
where $E_0 = \R X$ is the neutral bundle spanned by the geodesic vector field $X$ and $E_{+}$, $E_{-}$ are the stable and unstable bundles,  respectively (see e.g.\  \cite[p.\ 252]{Kni02}). 
Additionally, there is a smooth contact one-form $\alpha \in\Omega^1(S^*\mathcal M)$ which is simply the restriction of the Liouville one-form on $T^*\mathcal M$ to $S^*\mathcal M$. 
It fulfills
\[
 \iota_X\alpha =1,~~\ker(\alpha) = E_+\oplus E_-,~~ d\alpha \tu{ is symplectic on } \ker(\alpha),~~\mathcal L_X\alpha=0,
\]
where $\mathcal L_X$ denotes the Lie derivative. 
Note that the last two properties imply that $\alpha\wedge(d\alpha^n)$ is a nowhere-vanishing flow-invariant non-zero volume form which defines the Liouville measure on $S^*\mathcal M$.  
Using the contact one-form we get a splitting of the cotangent bundle into smooth subbundles
\[
 T^*(S^*\mathcal M) = \R\alpha\oplus X^\perp,~~~X^\perp:=\{(x,\xi)\in T^*(S^*\mathcal M): \xi(X(x)) = 0\}.
\]
We will call the smooth sections of $X^\perp$ perpendicular one-forms and denote their space by $\Omega^1_\perp(S^*\mathcal M)$. 
More generally, we introduce for $p=0,\ldots n$ the space of perpendicular $p$-forms
\[
 \Omega^p_\perp(S^*\mathcal M) := \{\omega\in \Omega^p(S^*\mathcal M): \iota_X\omega=0 \}=\Gamma^\infty(\Lambda^p X^\perp).
\]
By the Anosov splitting, the bundle $X^\perp$ can be further split into
\begin{equation}\label{eq:e_perp_splitting}
X^\perp = E_+^* \oplus E_-^* ,
\end{equation}
where the dual stable and unstable bundles are defined by $E_\pm^*(E_0\oplus E_\mp)=0$. 
In contrast to the smoothness of $X^\perp$, the subbundles $E^*_\pm$ are only H\"older continuous unless $\mathcal M$ is a locally symmetric space of rank one. 

More generally, we can consider an arbitrary Anosov vector field $Y$ on $S^\ast \M$ (again, see e.g.\  \cite[p.\ 252]{Kni02} for the definition), for which we have a splitting of the form \eqref{eq:Anosovsplitting} with $E_0=\R Y$ and we define the bundle
\[
Y^\perp:=\{(x,\xi)\in T^*(S^*\mathcal M): \xi(Y(x)) = 0\}.
\]

\begin{rem}[Complexifications]\label{rem:complex}When addressing spectral questions involving an operator on any of the bundles mentioned so far, or on any subbundle of a tensor power of $T^\ast(S^\ast \M)$, it is often more useful to work with the complexified bundle. For simplicity of notation we shall not explicitly distinguish in the following between real vector bundles and their complexifications. It will be clear from the context whether we refer to the real or the complexified bundle.\end{rem}

\subsection{Pollicott-Ruelle resonances on forms}\label{subsec:rp_res}
Pollicott-Ruelle resonances were introduced by Pollicott \cite{pol85} and Ruelle \cite{Rue86} in order to study mixing properties of hyperbolic flows (as mentioned before).
In the last years it has been found out that these resonances can also be defined as poles of meromorphically continued resolvents (see \cite{Liv04, BL07, GLP13}, \cite{FSj11, DZ16} for approaches using semiclassical analysis and \cite{DG16, BW17} for generalizations to noncompact settings). We follow \cite{DG16} to introduce the notion of Pollicott-Ruelle resonances on an arbitrary smooth complex vector bundle $\mathcal V \to S^*\mathcal M$.
For a vector field $Y\in \Gamma^\infty(T(S^\ast \M))$, a first order differential operator $\mathbf Y$ on $\mathcal V$ is called  \emph{admissible lift} of $Y$ if
\[
 \mathbf Y(f\mathbf u) = (Yf)\mathbf u + f\mathbf Y \mathbf u,~~~~f\in \Cinft(S^*\mathcal M), \;\mathbf u\in\Gamma^\infty(\mathcal V).
\]
This applies in particular to the geodesic vector field $X$, admissible lifts of which will be denoted by $\Xbf$.  
An example of an admissible lift of a vector field $Y$ is the Lie derivative $\mathcal L_Y$  on any $d\varphi_t$-invariant subbundle of $\otimes^p T^*(S^*\mathcal M)$ for some $p\in \N_0$ (taking into account Remark \ref{rem:complex}), where $\varphi_t$ is the flow of $Y$.  
In Section~\ref{sec:constant} we will additionally consider covariant derivatives which are further examples of admissible lifts. 
After choosing a smooth metric on $\mathcal V$ one defines the space $\L^2(S^*\mathcal M, \mathcal V)$. Note that by the compactness of $\mathcal M$ only the norm on this space depends on the choice of the metric but neither does the space nor its topology. Let now $Y$ be an Anosov vector field on $S^*\mathcal M$ and $\mathbf Y$ an admissible lift as above. Then one checks  \cite[Eq.\ (1.10)]{DG16} that there is a constant $C_{\mathbf Y}>0$ such that $\mathbf Y +\lambda:\L^2(S^*\mathcal M, \mathcal V) \to \L^2(S^*\mathcal M, \mathcal V)$
is invertible for $\tu{Re}(\lambda)>C_{\mathbf Y}$. The following statement was proved in the scalar case and for particular vector bundles in \cite{FSj11,DZ16,FT17} and is straightforward to adapt to the case of general vector bundles (see e.g.\ \cite[Thm.\ 1]{DG16}).
\begin{prop}
 The resolvent $R_{\mathbf Y,\mathcal V}(\lambda):=(\mathbf Y+\lambda)^{-1}:
 \L^2(S^*\mathcal M,\V)\to \L^2(S^*\mathcal M,\V)$,  $\tu{Re}(\lambda)\gg0$, has a continuation to the whole complex plane as a meromorphic family of bounded operators
 \[
  R_{\mathbf Y,\mathcal V}(\lambda): \Cinft(S^*\mathcal M,\mathcal V) \to \mathcal D'(S^*\mathcal M,\mathcal V).
 \]
Moreover, for any pole $\lambda_0$ the 
residue operators $\Pi_{\lambda_0} = \tu{res}_{\lambda=\lambda_0}(R_{\mathbf Y,\V}(\lambda))$ have finite rank. 
\end{prop}
\begin{definition}
 The poles of $R_{\mathbf Y,\V}(\lambda)$ are called  
 \emph{Pollicott-Ruelle resonances} of $\mathbf Y$. Given a resonance $\lambda_0$, 
 the finite-dimensional space $\Res_{\mathbf Y, \mathcal V}(\lambda_0):= 
 \tu{ran}(\Pi_{\lambda_0})\subset \mathcal D'(S^\ast\mathcal M, \mathcal V)$ is the space 
 of \emph{generalized Pollicott-Ruelle resonant states} and we call 
 $m_{\mathbf Y,\mathcal V}(\lambda_0):= \dim_\C \Res_{\mathbf Y, \mathcal V}(\lambda_0)$ 
 the \emph{multiplicity} of the resonance $\lambda_0$.
 
 If $\V=S^*\mathcal M\times \C$ is the trivial line bundle and $\mathbf Y=Y$, then we write just $\Res_{Y}(\lambda_0)$ and $m_{Y}(\lambda_0)$. 
\end{definition}
For any resonance $\lambda_0$ there exists a number $J(\lambda_0)\in\N$ such that the generalized resonant states have the following alternative description \cite[Theorem 2]{DG16}:
\begin{equation}\label{eq:WF_characterization}
 \Res_{\mathbf Y,\mathcal V}(\lambda_0) =\{u\in\mathcal D'(S^*\mathcal M,\mathcal V): (\mathbf Y+\lambda_0)^{J(\lambda_0)} u =0,~\tu{WF}(u)\subset E_+^*\}.
\end{equation}
If $J(\lambda_0) =1$ we say that the resonance \emph{has no Jordan block}. Otherwise, the space of \emph{Pollicott-Ruelle resonant states} $\tu{res}_{\mathbf Y,\mathcal V}(\lambda_0) := \ker(\mathbf Y+\lambda_0)\cap \Res_{\mathbf Y,\mathcal V}(\lambda_0)$ is a proper subspace of $\Res_{\mathbf Y,\mathcal V}(\lambda_0)$. 

Note that the co-sphere bundle $S^*\mathcal M$, the vector fields $Y$ on it (in particular, the geodesic vector field $X$), as well as their resolvents, Pollicott-Ruelle resonances, and associated resonant states and multiplicities depend on the Riemannian metric $\mathscr g$. 
In Section~\ref{sec:nonconst} we will be interested in their variation under perturbations of $\mathscr g$. 
For this reason we will write $S_{\mathscr g}^*\mathcal M$,  $X_{\mathscr g}$, $Y_{\mathscr g}$ in order to emphasize the dependence on $\mathscr g$. 
In the other sections we suppress  the Riemannian metric  in the notation. 
\section{Multiplicities on constant curvature manifolds}\label{sec:constant}
In this section we assume that $(\mathcal M,\mathscr g)$ is a closed orientable  hyperbolic\footnote{I.e., a Riemannian manifold of constant sectional curvature $-1$. Fixing the curvature at $-1$ is a common convention. By trivial rescaling arguments all results in this paper involving the resonance $0$ remain true if the metric is multiplied by a positive constant.}  manifold of dimension $n+1$.
\begin{prop}\label{prop:constant_mult_perp}
 If $n\neq 2$, then 
 \[
 m_{\mathcal L_X,X^\perp}(0) = b_1(\mathcal M).
 \]
 Furthermore, the resonance zero has no Jordan block, and if $n\geq 3$, then zero is the unique leading resonance  and there is a spectral gap.\footnote{I.e., there exists $\delta>0$ such that  $\mathcal L_X$ acting on $X^\perp$ has no resonances with real part in the interval $(-\delta,\infty)$ except the resonance zero.}
\end{prop}
The first part of this result will be a central ingredient for Theorem~\ref{thm:main2}. We will prove Proposition \ref{prop:constant_mult_perp} using a quantum-classical correspondence. Such correspondences have recently been developed in various contexts (see \cite{dfg} for compact hyperbolic manifolds \cite{GHW18, Had18} for the convex co-compact setting and \cite{GHW20}for generalizations to general rank one manifolds). We will use the general framework for vector bundles developed by the authors in \cite{kuester-weich18}. Additionally we use a Poisson transform due to Gaillard \cite{gaillard1} and combining both ingredients allows us to construct an explicit bijection between the Pollicott-Ruelle resonant states in perpendicular one forms and the kernel of the Hodge Laplacian. 
\begin{rem}\label{rem:exceptional}
The dimension $n+1=3$ is an exception where the multiplicity is given by $m_{\mathcal L_X, X^\perp}(0)= 2b_1(\mathcal M)$. The deeper reason for this exception is that  Gaillard's Poisson transform is not bijective in this case. The exceptional case could also be treated with our methods by a more detailed analysis of Gaillard's Poisson transform. This special case has however been worked out already in \cite[Proposition 7.7]{DGRS18} by factorizations of zeta functions, so we refrain from taking on the additional effort.  
\end{rem}
 
A crucial role in these quantum-classical correspondences is played by the so-called (generalized) first band resonant states
\begin{equation}\label{eq:first_band_states}
 \Res^{\mathrm{1st}}_{\mathbf X,\mathcal V}(\lambda_0) :=\Res_{\mathbf X,\mathcal V}(\lambda_0) \cap \ker\mathcal U_-,~~~~\tu{res}^{\mathrm{1st}}_{\mathbf X,\mathcal V}(\lambda_0) :=\tu{res}_{\mathbf X,\mathcal V}(\lambda_0) \cap \ker\mathcal U_-,
\end{equation}
where $\mathcal U_-$ is the horocycle operator which we will introduce below in \eqref{eq:horocycle_op}. 
Roughly speaking, first band resonant states are resonant states that are constant in the unstable directions. 
In the process of proving Proposition~\ref{prop:constant_mult_perp} we 
observe in Section \ref{prop:constant_mult_perp} that in any dimension $n+1$, including $n+1=3$, one has
\bq
\Res^{\mathrm{1st}}_{\mathcal L_X,X^\perp}(0)=\Res_{\mathcal L_X,X^\perp}(0),\label{eq:all1stband}
\eq
which means that all resonant states of the resonance zero are first band resonant states, even though for $n=1$  zero is not necessarily the leading resonance. Furthermore, we establish the following result:
\begin{prop}\label{prop:p_forms_Betti}
 On any closed orientable  hyperbolic manifold $\mathcal M$ of dimension $n+1$ and for any $p=0,\ldots, n$ with $p\neq n/2$, one has
\begin{equation}
\dim_\C \mathrm{Res}^{\mathrm{1st}}_{\mathcal L_X,\Lambda^p E^\ast_+}(0) =\dim_\C\mathrm{Res}^{\mathrm{1st}}_{\mathcal L_X,\Lambda^p E^\ast_-}(-2p)=b_p(\M). \label{eq:mainform}
\end{equation}
\end{prop}
We consider this result to be of independent interest  because it shows that also the higher Betti numbers can be recovered by considering Pollicott-Ruelle resonant states on certain vector bundles that are invariant under the horocycle transformation. Again the statement is obtained by constructing an explicit isomorphism onto the kernel of the Hodge Laplacian. 
\subsection{Description of the geometry of $\M$ in Lie-theoretic terms}\label{subsec:Lie}
Any closed orientable  connected hyperbolic manifold $\mathcal M$ of dimension $n+1$ can be written as a bi-quotient $$\mathcal M =\Gamma \backslash \mathbb{H}^{n+1}= \Gamma \backslash G/ K,$$ where $G=\SO(n+1,1)_0$,\footnote{Here the subscript $0$ indicates the identity component.} $K\cong\SO(n+1)$, and $\Gamma \subset G$ is a cocompact torsion-free discrete subgroup. $\mathcal M$ is thus an example of a Riemannian locally symmetric space of rank one. There exists a very efficient  Lie-theoretic language to describe the structure of $\mathcal M$, the co-sphere bundle $S^*\mathcal M$, as well as the invariant vector bundles which we introduce in this subsection.  For more details we refer the reader to 
\cite{GHW20,kuester-weich18} and for background information to the textbooks \cite{knapp,helgason78}. In the following we shall introduce the required abstract language in a quite concrete way, tailored to the particular group $G=\SO(n+1,1)_0$. 

The Lie algebra $\g=\so(n+1,1)$ of $G$ can be explicitly realized as a matrix algebra:
\bq\begin{split}
\g=\so(n+1,1)&=\Big\{\begin{pmatrix}k & p\\
p^T & 0 \end{pmatrix}:k\in \so(n+1),\; p\in \R^{n+1} \Big\}\\
&=\Big\{\begin{pmatrix}k & 0\\
0 & 0 \end{pmatrix}:k\in \so(n+1) \Big\}\oplus \Big\{\begin{pmatrix}0 & p\\
p^T & 0 \end{pmatrix}:p\in \R^{n+1} \Big\}\\
&=:\k \oplus \p,\label{eq:gmatrices}
\end{split}
\eq
where $\so(n+1)$ is the algebra of all real skew-symmetric matrices. The involution $\theta:\g\to\g$ given by $\theta X=-X^T$, $X\in \g$, is called Cartan involution.  The subspaces $\k$ and $\p$ are the eigenspaces of $\theta$ with respect to the eigenvalues $1$ and $-1$, respectively. $\k$ is the Lie algebra of the group $$K:=\exp(\k)\subset G,$$
where $\exp$ denotes the matrix exponential. We have $K\cong\SO(n+1)$. The splitting $\g=\k\oplus \p$ is called Cartan decomposition. This decomposition is $\Ad(K)$-invariant, where $\Ad(K)$ is the action of the matrix group $K$ on the matrix algebra $\k$  by conjugation. 

The tangent bundle $T\M=T(\Gamma\backslash G/K)$ can then be identified with the associated vector bundle $\Gamma\backslash G\times_{\Ad(K)} \p$, and similarly we identify $T^\ast\M= \Gamma\backslash G\times_{\Ad^\ast(K)} \p^\ast$, where $\Ad^\ast(K)$ is the dual representation of $\Ad(K)$. 

Via the Killing form  $\mathfrak{B}:\g\times \g\to \R$, which is given  explicitly by $\mathfrak{B}(X,Y)=2n\, \mathrm{tr}(X,Y)$, and the Cartan involution $\theta$ we define an $\Ad(K)$-invariant inner product $\eklm{\cdot,\cdot}$ on $\g$ by $$\eklm{X,Y}:=-(2n)^{-1}\mathfrak{B}(X,\theta Y)=\mathrm{tr}(XY^T),\qquad X,Y\in \g.$$  
The restriction of $\eklm{\cdot,\cdot}$ to $\p\times \p$ then defines a Riemannian metric of  constant curvature $-1$ on 
$\mathcal M$. We carry over the inner product to $\g^\ast$ using the isomorphism $\g\cong \g^\ast$ given by $X\mapsto \eklm{X,\cdot}$. 

We next want to describe the structure of the co-sphere bundle $S^*\mathcal M$ and the Anosov vector bundles $E_{0/+/-}$. To this end, we note that there is a maximal one-dimensional abelian subalgebra $\mathfrak{a} \subset \p$, given explicitly by 
\[
\aL=\Big\{\begin{pmatrix}0 & p\\
p^T & 0 \end{pmatrix}: p^T=(0,\ldots,0,H),H\in \R \Big\}\subset \g.
\]
We will denote the element in $\aL$ for which $H=1$ in the description above by $H_0$ and we identify $$\aL\cong\R$$ by mapping $H_0$ to $1$. Defining subspaces $\nL^\pm\subset \g$ by
\bq
\nL^\pm:=\Bigg\{\begin{pmatrix}0 & v & \mp v\\
-v^T & 0 & 0\\
\mp v^T & 0 &0 \end{pmatrix}: v\in \R^n \Bigg\},\label{eq:defn}
\eq
we see from \eqref{eq:gmatrices} that one has two decompositions
\[
\g=\k\oplus\aL\oplus\nL^+= \k\oplus\aL\oplus \nL^-.
\]
They are called Iwasawa decompositions. The spaces $\nL^\pm$ are  characterized by the property
\bq
[H_0,Y]=\pm Y\qquad\forall\; Y\in \nL^\pm,\label{eq:commrelroot}
\eq
and in fact they are the largest subspaces of $\g$ with these properties. In more abstract terms, the spaces $\nL^\pm$ are the root spaces with respect to the roots $\pm\alpha_0$, where $\alpha_0\in \aL^\ast$ is the element that maps $H_0$ to $1$. We will identify
\[
\nL^\pm\cong \R^n
\]
by mapping each matrix as in \eqref{eq:defn} to the vector $v$. 
Also on the group level there are two corresponding Iwasawa decompositions 
$G=KAN^+= KAN^-.$ 
 Here $N^\pm:=\exp(\nL^\pm)\subset G$ and $A:=\exp(\aL)\subset G$ are the matrix subgroups with Lie algebras $\nL^\pm$ and $\aL$, respectively.  For each group element $g\in G$ we now have unique Iwasawa ($+$) and opposite Iwasawa ($-$) decompositions
\begin{align}\begin{split}
g&=k^+(g)a^+(g)n^+(g)=k^+(g)\exp(H^+(g))n^+(g)\\
&=k^-(g)a^-(g)n^-(g)=k^-(g)\exp(H^-(g))n^-(g),\label{eq:iwasawadecomp}
\end{split}
\end{align}
where  $\exp(H^\pm(g))=a^\pm(g)$. In more concrete terms, this means that each matrix $g$ in $G$ can be written in a unique way as a product of three matrices in $K$, $A$, and $N^\pm$, respectively.  Assigning to each matrix in $G$ these unique matrices provides us with maps
\begin{equation}
k^\pm: G\to K,\qquad a^\pm: G\to A,\qquad H^\pm: G\to \aL,\qquad n^\pm: G\to N^\pm.\label{eq:Iwasawamaps}
\end{equation}
In addition, we define the group $$M:=\{m\in K:[m,a]=0\;\forall\; a\in A\}=\{m\in K:\Ad(m)(H)=0\;\forall\; H\in \aL\}\subset K$$ 
and let $\m$ be the Lie algebra of $M$. Explicitly, we have
\[
\m=\Bigg\{\begin{pmatrix}m & 0 & 0\\
0 & 0 & 0\\
0 & 0 &0 \end{pmatrix}: m\in \so(n) \Bigg\}\subset \k,\qquad M=\exp(\m)\cong \SO(n).
\]
The groups $N^\pm$ are normalized by $A$ and $M$. In fact, when identifying $\nL^\pm\cong \R^n$ as above, then the $\Ad(M)$-action on $\nL^\pm\cong \R^n$ is just the defining representation of $\SO(n)$ on $\R^n$. 
We have the so-called Bruhat decomposition 
\begin{equation}
\g=\m\oplus \aL\oplus  \nL^+\oplus\nL^- \label{eq:bruhat}
\end{equation}
which turns out to be invariant under the $\Ad(M)$-action. 

The co-sphere bundle $S^*\mathcal M$ can be identified with $\Gamma \backslash G/M$. Indeed, the element $\alpha_0\in \aL^\ast\subset \p^\ast$ introduced above fulflls $\norm{\alpha_0}=1$ and
\[
\Gamma \backslash G/M\owns \Gamma gM\mapsto [\Gamma g,\alpha_0]\in S^*\mathcal M\subset T^\ast\M=\Gamma\backslash G\times_{\Ad^\ast(K)} \p^\ast
\]
is a well-defined diffeomorphism.  The Lie group $A\cong \R$ acts from the right on $\Gamma \backslash G/M$ because it commutes by definition with $M$, and this action precisely coincides with the geodesic flow. In particular, the geodesic vector field $X\in \Gamma^\infty(T(S^\ast \M))$ corresponds to the constant function $\bar X:G\to \aL$ with $\bar X(g)=H_0$ for all $g\in G$. Furthermore, the tangent bundle of $S^*\mathcal M$ can be identified as follows:
\begin{equation}
 \label{eq:splitting_associated}
T(S^*\mathcal M) = \Gamma\backslash G\times_{\Ad(M)} (\aL\oplus\nL^+ \oplus\nL^-)=\R X\oplus \underbrace{\Gamma\backslash G\times_{\Ad(M)}\nL^+}_{=E_+}\oplus \underbrace{\Gamma\backslash G\times_{\Ad(M)}\nL^-}_{=E_-}.
\end{equation}
There is an analogous identification of $T^*(S^*\mathcal M)$. The Anosov stable and unstable bundles $E_\pm$ can be  described more concretely using their lifts $\widetilde E_\pm$ to the frame bundle $F\M=\gam G$ along the $M$-orbit projection $F\M=\gam G\to \gam G/M=S\M$: Choosing an orthonormal basis $U^\pm_1,\ldots,U^\pm_n$ of $\nL^\pm$, the constant function $G\to \nL^\pm$ with value $U^\pm_j$ defines a nowhere-vanishing vector field on $F\mathcal M$, denoted also by $U^\pm_j$, and one has
\bq
\widetilde E_\pm=\mathrm{span}_{\R}(U^\pm_1,\ldots,U^\pm_n).\label{eq:liftedEpm}
\eq
The boundary at infinity of the hyperbolic space $\mathbb{H}^{n+1}=G/K$ is diffeomorphic to the sphere $S^{n}$ and can be realized as
\[
\partial_{\infty}\mathbb{H}^{n+1}=K/M=\SO(n+1)/\SO(n)\cong S^{n}.
\]
Consequently, the tangent bundle of $\partial_{\infty}\mathbb{H}^{n+1}$ can be identified with 
\[
T(K/M)=K\times_{\mathrm{Ad}(M)}\m^{\perp_\k},
\]
where $\m^{\perp_\k}\subset \k$ denotes the orthogonal complement of $\m$ in $\k$, given explicitly by
\[
\m^{\perp_\k}=\Bigg\{\begin{pmatrix}0 & v & 0\\
-v^T & 0 & 0\\
0 & 0 &0 \end{pmatrix}: v\in \R^n \Bigg\}.
\]
We can identify $\m^{\perp_\k}\cong \R^n$ by mapping each matrix as above to $v$. The restriction of the representation $\Ad(M)$ to $\m^{\perp_\k}$ is then just the defining representation of $\SO(n)$ on $\R^n$. 

In view of these identifications all vector bundles over $S^*\mathcal M$ of interest in the following are associated vector bundles of the form $\V_\tau:=G\times_\tau V$ with respect to some finite-dimensional complex $M$-representation $(\tau,V)$. 

As all our homogenous spaces are reductive there always exists a canonical connection that we denote by
\begin{equation}
\nabla: \Gamma^\infty(\V_\tau)\to \Gamma^\infty(\V_\tau\otimes T^\ast(S^*\mathcal M)).
\end{equation}
To describe how $\nabla$ is defined, let us regard a section $s\in \Gamma^\infty(\V_\tau)$ as a right-$M$-equivariant function $\bar s\in \Cinft(\Gamma\backslash G,V)$. Moreover, by \eqref{eq:splitting_associated} we regard a vector field $\X\in \Gamma^\infty(T(S^*\mathcal M))$ as a right-$M$-equivariant function $\bar \X\in \Cinft(\Gamma\backslash G,\nL^+\oplus \aL\oplus {\nL^-})$, that is, $\bar \X(\Gamma gm)=\Ad(m^{-1})\bar \X(\Gamma g)$ for every $m\in M$. Then $\nabla$ is defined by the covariant derivative
\begin{align}\begin{split}
\nabla_{\X}(s)(\Gamma gM):=\frac{d}{dt}\Big|_{t=0}\bar s\big(\Gamma ge^{t\bar \X(\Gamma g)}\big).\label{eq:covariantderiv}
\end{split}
\end{align}

\subsection{Horocycle operators}\label{sec:horocycle}
Horocycle operators have been introduced in \cite{dfg} as a crucial tool for establishing quantum-classical correspondences. We already mentioned them in the  definition of the first band resonant states \eqref{eq:first_band_states} in the introduction. 
They are defined as follows: Let $(\mathcal V,\nabla)$ be a vector bundle over $S^\ast \M$ with a connection $\nabla$ and denote by $\widetilde {\tu{pr}}_{E^\ast_-}:\Gamma^\infty(\mathcal V\otimes T^\ast (S^\ast \M))\to \Gamma^\infty(\mathcal V\otimes E^\ast_-) $ the map induced by the fiber-wise orthogonal projection $\tu{pr}_{E^\ast_-}:T^*(S^*\mathcal M) \to E^*_-$ onto the subbundle $E^*_-\subset T^*(S^*\mathcal M)$. Then 
we define the horocyle operator $\mathcal U_-$ of $(\mathcal V,\nabla)$  by composing the connection $\nabla: \Gamma^\infty(\mathcal V) \to \Gamma^\infty(\mathcal V\otimes T^\ast (S^\ast \M))$ with $\widetilde {\tu{pr}}_{E^\ast_-}$:
\begin{equation}\label{eq:horocycle_op}
\mathcal U_-:= \widetilde{\tu{pr}}_{E^\ast_-} \circ \nabla: 
 \Gamma^\infty(\mathcal V) \to \Gamma^\infty(\mathcal V\otimes E^\ast_-).
\end{equation}
By duality, $\mathcal U_-$ extends to distributional sections. In the concrete language of \eqref{eq:liftedEpm} we can express $\mathcal U_-$ as follows: If $\widetilde\mathcal V=\pi^\ast \mathcal V$ is the lift of $\mathcal V$  to the frame bundle, i.e., the pullback bundle with respect to the $M$-orbit projection $\pi:F\M=\gam G\to \gam G/M=S\M$ and if $\tilde u\in \Gamma^\infty(\widetilde\mathcal V)$ is the lift of a section $u\in \Gamma^\infty(\mathcal V)$, then the lift of the section $\mathcal U_- u$ to the bundle $ \widetilde{\mathcal V\otimes E^\ast_-}\cong \widetilde\mathcal V\otimes \widetilde E^\ast_-$ is given by
\[
\widetilde{\mathcal U_- u}=\sum_{j=1}^n\widetilde \nabla_{U_j^-}\tilde u\otimes (U^-_j)^\ast,
\]
where $(U^-_j)^\ast\in \Gamma^\infty(\widetilde E_-^\ast)$ is the dual vector field of $U^-_j$ and $\widetilde \nabla=\pi^\ast  \nabla$ the lifted (i.e., pullback) connection on $\widetilde \V$. 

As already stated in \eqref{eq:first_band_states}, the so-called first band resonant states are defined as those resonant states that are annihilated by $\mathcal U_-$.  
The main technical feature of $\mathcal U_-$ is that it obeys the commutation relation
\bq
 \nabla_X\mathcal U_- - \mathcal U_- \nabla_X  =\mathcal U_-.\label{eq:commUX}
\eq
   This is a consequence of the   commutation relations  \eqref{eq:commrelroot}, the definition \eqref{eq:covariantderiv} of the covariant derivative, and the observation from Section \ref{subsec:Lie} that the geodesic vector field $X$ corresponds to the constant function $\bar X:G\to \aL$ with value $H_0$. If $u\in \Res_{\nabla_X,\mathcal V}(\lambda)$ for some $\lambda\in \C$ and $J\in \N$ is such that $(\nabla_X+\lambda)^Ju=0$, then  \eqref{eq:commUX} implies
\bqn
(\nabla_X+\lambda)^J\mathcal U_-u=(\nabla_X+\lambda)^{J-1}(\nabla_X+\lambda)\mathcal U_-u=(\nabla_X+\lambda)^{J-1}\mathcal U_-(\nabla_X+\lambda+1)u=\cdots=\mathcal U_-(\nabla_X+\lambda+1)^J u,
\eqn 
which proves the following very useful \emph{shifting property} of the horocycle operator $\mathcal U_-$:
\bq
\mathcal U_-\big(\mathrm{Res}_{\nabla_X,\mathcal V}(\lambda)\big)\subset \Res_{\nabla_X,\mathcal V\otimes  E^\ast_-}(\lambda+1),\qquad \mathcal U_-\big(\mathrm{res}_{\nabla_X,\mathcal V}(\lambda)\big)\subset \mathrm{res}_{\nabla_X,\mathcal V\otimes  E^\ast_-}(\lambda+1).\label{eq:Ushift}
\eq

\subsection{First band resonant states and principal  series representations}
As already mentioned above, the homogeneous space $K/M \cong S^{n}$ can be regarded as the \emph{boundary at infinity} of the Riemannian symmetric space $G/K = \mathbb H^{n+1}$ and using the Iwasawa projection we can define a left-$G$-action 
\begin{equation}
g(kM):=k^-(g k)M,\qquad g \in G,\;k\in K.\label{eq:GactionsbaseK}
\end{equation}
Given a finite-dimensional complex $M$-representation $(\tau, V)$ we define the \emph{boundary vector bundle} 
\[
{\V^{\mathcal{B}}_{\tau}} =(K\times_{\tau} V,\pi_{\V^{\mathcal{B}}_{\tau}}),\qquad \pi_{\V^{\mathcal{B}}_{\tau}}([k,v])=kM. 
\]
The total space $K\times_{\tau} V$ of $\V^{\mathcal{B}}_{\tau}$ carries the $G$-action
\begin{equation} 
g[k,v]:=[k^-(gk),v],\qquad g\in G,\;k\in K,\label{eq:Gactionstotal}
\end{equation}
that lifts the $G$-action \eqref{eq:GactionsbaseK} on the base space $K/M$. Consequently, we get an induced action on smooth sections:
\begin{equation}
(g s)(kM):={g}\big(s\big(g^{-1}(kM)\big)\big),\qquad s\in \Gamma^\infty(\V^{\mathcal{B}}_{\tau}),\;g\in G.\label{eq:Gactionstotalsections1}
\end{equation}
If we consider a section $s\in \Gamma^\infty(\V^{\mathcal{B}}_{\tau})$ as a right-$M$-equivariant smooth function $\bar s:K\to V$, the action \eqref{eq:Gactionstotalsections1} corresponds to assigning to $\bar s$ for any $g\in G$ the right-$M$-equivariant smooth function $\overline{g s}:K\to V$ given by
\begin{equation}
\overline{gs}(k)=\bar s(k^-(g^{-1}k)),\qquad g\in G,\;k\in K.\label{eq:Gactionsectionslanglands}
\end{equation}
To describe how the principal series representation of $G$ associated to an $M$-representation $\tau$ and a parameter  $\lambda \in \C$ acts on smooth sections of $\mathcal V_\tau^{\mathcal B}$, let us regard a section $s\in \Gamma^\infty(\mathcal V_\tau^{\mathcal B})$ as a right-$M$-equivariant function $\bar s\in \Cinft(K,V)$. 
We then set\footnote{We use a simplified (non-standard) notation and follow Olbrich's convention as in \cite[between Satz 2.8 and Satz 2.9]{olbrichdiss}. In \cite[p.\ 169]{knapp2}, the definition differs from ours in such a way that $\lambda$ is replaced by $-\lambda$.
Furthermore recall that we identified $\aL\cong \R$ in Section \ref{subsec:Lie}.}
\begin{equation}
\overline{\pi^{\lambda}_{\tau}(g)s}(k):=e^{(\lambda + n/2)H^-(g^{-1}k)}\bar s(k^{-}(g^{-1}k))\label{eq:compactp},\quad s\in \Gamma^\infty(\V^\B_\tau),\; kM\in K/M.
\end{equation}
This representation extends by continuity to a representation $\pi^{\lambda}_{\tau}:G\to \mathrm{End}(\D'(K/M,\V^{\mathcal{B}}_{\tau}))$.
One has the following important relation between first band resonant states and  the $\Gamma$-invariant distributional sections of the boundary vector bundle with respect to the principal series representation $\pi^{-\lambda-n/2}_\tau$.
\begin{prop}[{\cite[Lemma 2.15]{kuester-weich18}}]\label{prop:boundary_distribution}
For each $\lambda\in \C$ there is an explicit isomorphism
\begin{equation}
Q_{\lambda}:\tu{res}^{\mathrm{1st}}_{\nabla_X,\V_{\tau}}(\lambda)\stackrel{\cong}{\longrightarrow}{^\Gamma}\big(\D'(K/M,\V^{\mathcal{B}}_{\tau}), \pi^{-\lambda-n/2}_\tau\big)\label{eq:ruellerepiso}
\end{equation}
onto the space of all distributional sections $u$ of $\V^{\mathcal{B}}_{\tau}$ with $ \pi^{-\lambda-n/2}_\tau(\gamma)u=u$ for every $\gamma\in \Gamma$.
\end{prop}
\subsection{Relating resonances of the Lie- and covariant derivatives}
Proposition~\ref{prop:boundary_distribution} provides a powerful way to handle first band resonant states of the covariant derivative  $\nabla_X$ along the geodesic vector field. 
In Proposition~\ref{prop:constant_mult_perp} and \ref{prop:p_forms_Betti} we are however interested in resonant states of the Lie derivative. 
Therefore we have to relate these states: 
\begin{lem}\label{lem:shift451}For $p\in\{0,1,2,\ldots\}$, suppose that $\tau$ is a subrepresentation of $\otimes^p(\mathrm{Ad}(M)|_{\nL^\pm})$. 
Then the covariant derivative and the Lie derivative along the geodesic vector field $X$, acting on smooth sections of $\V_{\tau}$, are related by
\[
\mathcal{L}_{X} = \nabla_{X} \mp p\,\mathrm{id}_{\Gamma^\infty(\V_{\tau})}.
\] 
Consequently, one has for every $\lambda \in \C$ and $p\in \N$
\begin{equation}\label{eq:reslres}
\mathrm{Res}_{\mathcal L_X,\V_\tau}(\lambda)=\mathrm{Res}_{\nabla_X, \V_\tau}(\lambda \mp p) ~~\tu{ and }~~ \mathrm{res}_{\mathcal L_X,\V_\tau}(\lambda)=\mathrm{res}_{\nabla_X, \V_\tau}(\lambda \mp p).
\end{equation}
\end{lem} 
\begin{proof}
Recall that the geodesic flow on $S^\ast(\gam G/K)=\gam G/M$ is given by
\begin{equation}
\varphi_t(\Gamma gM)=\Gamma ge^{tH_0}M,\qquad t\in \R.\label{eq:flow35902}
\end{equation}
Its derivative $d\varphi_t:T(\gam G/M)=\gam G\times _{\mathrm{Ad}(M)}(\nL^+\oplus \aL\oplus \nL^-)\to \gam G\times _{\mathrm{Ad}(M)}(\nL^+\oplus \aL\oplus \nL^-)$ reads
\begin{equation}
d\varphi_t(\Gamma gM)([\Gamma g M, v])=[\Gamma g M, \mathrm{Ad}(e^{-tH_0})v],\qquad t\in \R,\;[\Gamma g M, v]\in \gam G\times _{\mathrm{Ad}(M)}(\nL^+\oplus \aL\oplus \nL^-).\label{eq:flowderiv}
\end{equation}
Any vector $v\in \nL^\pm$ is an eigenvector of the adjoint action:
\begin{equation}
\mathrm{Ad}(e^{-tH_0})v =e^{-t\mathrm{ad}(H_0)}v =e^{\mp t} v.\label{eq:adactions}
\end{equation}
Let now $\omega\in \Gamma^\infty(\V_\tau)$, identified with a left-$\Gamma$-, right-$M$-equivariant function $\overline \omega: G\to V$, where $V\subset\otimes^p (\nL^\pm)$.
Considering $\varphi_t$ as a left-$\Gamma$-, right-$M$-equivariant map $\bar \varphi_t: G\to G$, let $\overline{\varphi_t^\ast \omega}: G\to V$ be the left-$\Gamma$-, right-$M$-equivariant function corresponding to $\varphi_t^\ast \omega\in \Gamma^\infty(\V_\tau)$. 
Then we get with \eqref{eq:adactions} for  $g\in G$ and $v_1,\ldots,v_p\in \nL^\pm$:
\[
\overline{\varphi_t^\ast \omega}(g)(v_1,\ldots,v_p)
=\bar \omega(ge^{tH_0})(e^{\mp t }v_1,\ldots,e^{\mp t }v_p)=e^{\mp p t}\bar \omega(ge^{tH_0})(v_1,\ldots,v_p).
\]
For the Lie derivative of $\omega$ we then obtain with the analogous ``$\,\bar{\;}\bar{\;}\,$''-notation and the product rule
\begin{align*}
\overline{\mathcal{L}_{X}\omega}(g)(v_1,\ldots,v_p)&=\frac{d}{dt}\Big|_{t=0}\overline{\varphi_t^\ast \omega}(g)(v_1,\ldots,v_p)\\
&=\frac{d}{dt}\Big|_{t=0}\Big(e^{\mp p t }\bar \omega(ge^{tH_0})(v_1,\ldots,v_p)\Big)\\
&=\frac{d}{dt}\Big|_{t=0}\bar \omega(ge^{tH_0})(v_1,\ldots,v_p) \mp p \bar\omega(g)(v_1,\ldots,v_p) \\
&=\overline{\nabla_{X}\omega}(g)(v_1,\ldots,v_p) \mp p \bar\omega(g)(v_1,\ldots,v_p).
\end{align*}
Here we recalled the definition \eqref{eq:covariantderiv} of the canonical covariant derivative.
\end{proof}
\subsection{Proof of Proposition~\ref{prop:p_forms_Betti}}
Let us collect what we have obtained so far: By Lemma~\ref{lem:shift451} 
\[
\mathrm{res}^{\mathrm{1st}}_{\mathcal L_X,\Lambda^p E^\ast_+}(0) = \mathrm{res}^{\mathrm{1st}}_{\nabla_X,\Lambda^p E^\ast_+}(-p)
~~\tu{ and }~~
\mathrm{res}^{\mathrm{1st}}_{\mathcal L_X,\Lambda^p E^\ast_-}(-2p) = \mathrm{res}^{\mathrm{1st}}_{\nabla_X,\Lambda^p E^\ast_-}(-p).
\]
As the adjoint action of $M$ on $\nL^\pm$ is given by the defining representation of 
$\SO(n)$ on $\R^n$ we deduce from \eqref{eq:splitting_associated} that  
$\Lambda^p(E^*_\pm) = \Gamma\backslash G\times_{\tau_p} \Lambda^p(\R^n)$ with $\tau_p$
being the p-th exterior power of the standard action of $\SO(n)$ on $\R^n$. 
By Proposition~\ref{prop:boundary_distribution} we can thus identify
\[
 \mathrm{res}^{\mathrm{1st}}_{\mathcal L_X,\Lambda^p E^\ast_+}(0)
 \cong 
 \mathrm{res}^{\mathrm{1st}}_{\mathcal L_X,\Lambda^p E^\ast_-}(-2p) 
 \cong 
 {^\Gamma}\big(\D'(K/M,\V^{\mathcal{B}}_{\tau_p}), \pi^{p-n/2}_{\tau_p}\big).
\]
We now use a vector-valued Poisson transform. To this end, let  
$\Delta_H = d\delta + \delta d$ be the Hodge Laplacian on $\Omega^p(\mathbb H^{n+1})$. 
\begin{thm}[Poisson transform for $\Gamma$-invariant $p$-forms]\label{thm:gaillard}
 Let $K=\SO(n+1)$, $M=\SO(n)$, and let $\tau_p$ be the $p$-th exterior power of the defining representation of $\SO(n)$ on $\R^n$. Then for any $\lambda\in\C$ with $\lambda\neq n-p$ and $\lambda\neq n+1, n+2,\ldots$, there is an isomorphism of vector spaces
 \bqn
P_{\tau_p,\lambda}: {^\Gamma}\big(\D'(K/M,\V^{\mathcal{B}}_{\tau_p}), \pi^{\lambda-n/2}_{\tau_p}\big)  \to\big\{\omega \in \Omega^p(\M): \Delta_H \omega = (\lambda-p)(n-\lambda - p)\omega,~\delta \omega =0\big\}.
 \eqn
\end{thm}
This result is due to Gaillard (see \cite[Thm.\ 2' c) and Thm.\ 3']{gaillard1}, taking into account that $\Gamma$-invariant smooth forms are trivially \emph{slowly growing} in Gaillard's sense because $\Gamma$ is co-compact) although it requires some work (see Section~\ref{subsec:gaillard}) to translate his statements into the form stated above that we can apply in our setting. For $p\neq n/2$ the Poisson transform $P_{\tau_p,p}$ 
is bijective and thus
\[
 {^\Gamma}\big(\D'(K/M,\V^{\mathcal{B}}_{\tau_p}), \pi^{p-n/2}_{\tau_p}\big) \cong\left\{\omega\in\Omega^p(\M), \Delta_H\omega=0,\delta \omega=0\right\}.
\]
As on compact manifolds any harmonic form is co-closed, the right hand side is simply the kernel of the Hodge Laplacian and Hodge theory implies that its dimension equals the $p$-th Betti number of $\M$. We thus have shown
\[
\dim \mathrm{res}^{\mathrm{1st}}_{\nabla_X,\Lambda^p E^\ast_+}(-p) = \dim \mathrm{res}^{\mathrm{1st}}_{\mathcal L_X,\Lambda^p E^\ast_-}(-2p) = 
b_p(\mathcal M).
\]
Now using once more that $p\neq n/2$ \cite[Theorem 6.2]{kuester-weich18} implies that the resonance at $-p$ of $\nabla_X$ has no Jordan block and consequently
\begin{equation}\label{eq:dimensions_first_band}
  \dim \mathrm{Res}^{\mathrm{1st}}_{\mathcal L_X,\Lambda^p E^\ast_+}(0) =\dim \mathrm{Res}^{\mathrm{1st}}_{\mathcal L_X,\Lambda^p E^\ast_-}(-2p)= \dim \mathrm{Res}^{\mathrm{1st}}_{\nabla_X,\Lambda^p E^\ast_+}(-p)=\dim \mathrm{res}^{\mathrm{1st}}_{\nabla_X,\Lambda^p E^\ast_+}(-p) =b_p(\M).
\end{equation}
This finishes the proof of Proposition~\ref{prop:p_forms_Betti}.

\subsection{Proof of Proposition~\ref{prop:constant_mult_perp}}
Let $\lambda\in \C$. By the decomposition \eqref{eq:e_perp_splitting} and Lemma \ref{lem:shift451}, we have
\[
\Res_{\mathcal L_X,X^\perp}(\lambda)\cong \mathrm{Res}_{\mathcal L_X,E^\ast_+}(\lambda)\oplus \mathrm{Res}_{\mathcal L_X,E^\ast_-}(\lambda)=\mathrm{Res}_{\nabla_X, E^\ast_+}(\lambda-1)\oplus \mathrm{Res}_{\nabla_X,E^\ast_-}(\lambda+1).
\]
As $\nabla_X$ is an antisymmetric operator in $\L^2(E^\ast_-)$ there are no resonances of $\nabla_X$ on $E_-$ 
with positive real part\footnote{\label{foot:antisymm}Since the geodesic flow preserves the Liouville measure on $S^\ast \M$ and the norm on the bundle $E_-$, one can show that $\nabla_X$ is antisymmetric in $\L^2$ and one can write down an explicit formula for the $\L^2$-resolvent $(\nabla_X+\lambda)^{-1}$ when $\Re \lambda>0$, see e.g.\ \cite[(1.10)]{DG16}.}, so if $\Re \lambda >-1$ one has
\bq
  \Res_{\mathcal L_X,X^\perp}(\lambda)\cong \mathrm{Res}_{\nabla_X, E^\ast_+}(\lambda-1).\label{eq:resE+}
\eq
By the definition of first band resonant states \eqref{eq:first_band_states} 
and the dimension formula for linear maps we conclude
\begin{equation}
\dim\Res_{\nabla_X, E^\ast_+}(\lambda-1) = \dim\Res^{\mathrm{1st}}_{\nabla_X, E^\ast_+}(\lambda-1) + 
\dim\mathcal{U}_-\big(\Res_{\nabla_X, E^\ast_+}(\lambda-1)\big).\label{eq:3859210890}
\end{equation}
Regarding the statement on the leading resonance, we note that if $n\geq 3$ and $\Re \lambda >-1$, then by  Proposition \ref{prop:boundary_distribution} and  Theorem \ref{thm:gaillard} there is an isomorphism
\bq
\mathrm{res}^{\mathrm{1st}}_{\nabla_X,E^\ast_+}(\lambda-1) \cong \{\omega\in\Gamma^\infty(T^\ast\M): \Delta_H\omega=-\lambda(n+\lambda - 2)\omega, ~\delta \omega =0\},\label{eq:298090}
\eq
where $\Delta_H$ is the Hodge Laplacian on $\M$. 
 When $\Re \lambda > 1-\frac{n}{2}$, the  eigenvalue $-\lambda(n+\lambda - 2)$ is real and positive iff $\lambda\in(1-\frac{n}{2},0]$ and if this does not hold the right hand side of \eqref{eq:298090} is the zero space. It follows for  $n\geq 3$ and $\Re \lambda > 1-\frac{n}{2}$ that $\Res^{\mathrm{1st}}_{\nabla_X, E^\ast_+}(\lambda-1)=\{0\}$ unless $\lambda\in (1-\frac{n}{2},0]$ because every Jordan block would contain at least one resonant state. 
Now, in view of Proposition~\ref{prop:p_forms_Betti},   \eqref{eq:dimensions_first_band}, and \eqref{eq:3859210890},  
it remains to prove $\mathcal{U}_-(\Res_{\nabla_X, E^\ast_+}(\lambda-1)) = 0$ under the assumption that $n \neq 2$ and  $\Re \lambda > -\delta$ for some small $\delta>0$ to establish Proposition~\ref{prop:constant_mult_perp}. Recall from \eqref{eq:horocycle_op} that  
$\mathcal{U}_-(\Res_{\nabla_X, E^\ast_+}(\lambda-1)) \subset\mathcal D'(\mathcal M, E^*_+\otimes E^ *_-)$. Further, by \eqref{eq:Ushift} one has 
\[
 \mathcal{U}_-\big(\Res_{\nabla_X, E^\ast_+}(\lambda-1)\big) \subset \Res_{\nabla_X, E^\ast_+\otimes E^ \ast_-}(\lambda).
\]
If $\Re \lambda >0$, we immediately get the zero space on the right hand side as otherwise there would be resonances of $\nabla_X$ with positive real part, which is impossible by the antisymmetry of $\nabla_X$ in $\L^2(E^\ast_+\otimes E^ \ast_-)$, cf.\ Footnote \ref{foot:antisymm}.  
We are left with the proof of $\mathcal{U}_-(\Res_{\nabla_X, E^\ast_+}(\lambda-1)) = 0$ for $\Re \lambda \in (-\delta,0]$ with some small $\delta>0$. 
Another application of \eqref{eq:Ushift} and the absence of resonances of $\nabla_X$ with positive real part due to  antisymmetry  implies $$\Res_{\nabla_X, E^\ast_+\otimes E^ \ast_-}(\lambda) = \Res^ {\mathrm{1st}}_{\nabla_X, E^\ast_+\otimes E^ \ast_-}(\lambda)\qquad\text{if }\Re \lambda> -1.$$
Using the quantum-classical correspondence once more we shall obtain a simple description of the latter spaces. To this end, note that
the Cartan involution $\theta|_{\nL^+}:\nL^+\to \nL^-$ is an equivalence of representations $\Ad(M)|_{\nL^+}\sim \Ad(M)|_{\nL^-}$ which induces an isomorphism $E^\ast_+\cong E^\ast_-$ that is compatible with the connections on the two bundles. This in turn induces a connection-compatible isomorphism $E^\ast_+\otimes  E^\ast_-\cong E^\ast_-\otimes  E^\ast_-$.  As the covariant derivatives $\nabla_X$ as well as the horocycle operators $\mathcal U_-$ are defined in terms of the respective connections, we conclude 
\[
\Res^{\mathrm{1st}}_{\nabla_X, E^\ast_+\otimes  E^\ast_-}(\lambda)\cong\Res^{\mathrm{1st}}_{\nabla_X, E^\ast_-\otimes  E^\ast_-}(\lambda).
\]
Now let $\tilde{\mathscr g}$ be the Riemannian metric on $S^\ast\mathcal M$ induced by the Sasaki metric on $T^\ast \M$ with respect to the Riemannian metric on $\M$. The  restriction of $\tilde{\mathscr g}$ to $E_-\times E_-$ defines a smooth section of $E^\ast_-\otimes  E^\ast_-$. 

If $n=2$, then $\Lambda^2E^\ast_-\subset E^\ast_-\otimes  E^\ast_-$ is the top-degree exterior power of $E_-$ and hence trivialized by  choosing an orientation form $ \Omega_{E_-}$ on $E_-$. Choosing a non-zero element $\Omega_0\in \Lambda^2\nL^\ast$, we can define $ \Omega_{E_-}$ to be the smooth section of $\Lambda^2E^\ast_-=\gam G\times_{\Lambda^2 \Ad^\ast(M)}(\Lambda^2\nL^\ast)$ induced by the constant function $G\to \Lambda^2\nL^\ast$ with the value $\Omega_0$.

\begin{lem}\label{lem:sym_tensor_resonances}There is a number $\delta>0$ such that for all $\lambda \in \C$ with $\Re \lambda \in (-\delta,0]$ one has
 \[
\Res^ {\mathrm{1st}}_{\nabla_X, E^\ast_-\otimes E^ \ast_-}(\lambda)  =\begin{cases}\{c\, \tilde{\mathscr g}|_{E_-\times E_-}: c:S^\ast\mathcal M\to\C \tu{ locally constant}\},\qquad & \lambda=0,n\neq 2\\
\{c\, \tilde{\mathscr g}|_{E_-\times E_-} + \tilde  c\, \Omega_{E_-}: c,\tilde c:S^\ast\mathcal M\to\C \tu{ locally constant}\},\qquad & \lambda=0,n=2\\
\{0\}, &\text{else}.\end{cases}
 \]
\end{lem}
Before proving this lemma let us see how it finishes the proof of Proposition~\ref{prop:constant_mult_perp} and \eqref{eq:all1stband}: All that is left to prove is that if $\mathcal U_- s=c \eta$ with $\eta\in\{\tilde{\mathscr g}|_{E_-\times E_-},\Omega_{E_-}\}$, $s\in \Res_{\nabla_X, E^\ast_+}(-1)$, and $c\in \C$, then $c=0$. This is easy:\footnote{We thank Colin Guillarmou for suggesting the slick argument.} If $\mathcal U_- s=c \eta$, then
\[
\big<\mathcal U_- s,\eta\big>_{\L^2(S^\ast\M,S^2(E^\ast_-))}=c \Vert\eta\Vert_{\L^2(S^\ast\M,S^2(E^\ast_-))}^2.
\]
Thus, if $\mathcal U^\ast_-$ is the formal adjoint of $\mathcal U_-$, we have
\begin{equation}
s(\mathcal U^\ast_-(\eta))=c \Vert\eta\Vert_{\L^2(S^\ast\M,S^2(E^\ast_-))}^2,\label{eq:final}
\end{equation}
where the left hand side is the pairing of the distributional section $s$ with the smooth section $\mathcal U^\ast_-(\eta)$. In \cite[Lemma 4.3]{dfg} it is shown that $\mathcal U^\ast_-=-\mathcal T\circ \mathcal U_-$, $\mathcal T$ being the trace operator. The smooth  section $\eta$ vanishes under all covariant derivatives as it corresponds to the constant function $G\to \nL^\ast\otimes\nL^\ast$ with either the value $\eklm{\cdot,\cdot}|_{\nL\times\nL}$ or the value $\Omega_0$. Therefore, we find $\mathcal U^\ast_-(\eta)=0$ and \eqref{eq:final} implies $c=0$.

It remains to prove Lemma~\ref{lem:sym_tensor_resonances}:
\begin{proof}[Proof of Lemma~\ref{lem:sym_tensor_resonances}]
The tensor product $E^\ast_-\otimes  E^\ast_-$ splits into a sum of three subbundles according to
\[
E^\ast_-\otimes  E^\ast_-=S^2_0(E^\ast_-)\oplus \Lambda^2E^\ast_-\oplus \C \tilde{\mathscr g}|_{E_-\times E_-},
\]
where $S^2_0(E^\ast_-)$ denotes the trace-free symmetric tensors of rank $2$. Note that $\C \tilde{\mathscr g}|_{E_-\times E_-}$ is a trivial line bundle and for $n=1$ the other two bundles have rank zero. By the additivity of resonance multiplicities with respect to Whitney sums of vector bundles, we arrive at
\begin{equation}
\Res^{\mathrm{1st}}_{\nabla_X, E^\ast_-\otimes  E^\ast_-}(\lambda)\cong\Res^{\mathrm{1st}}_{\nabla_X,S^2_0(E^\ast_-)}(\lambda)\oplus \Res^{\mathrm{1st}}_{\nabla_X,\Lambda^2E^\ast_-}(\lambda)\oplus \Res^{\mathrm{1st}}_{\nabla_X,\C \tilde{\mathscr g}|_{E_-\times E_-}}(\lambda).\label{eq:3693sf9sgs0}
\end{equation}
Now we can consider the three summands on the right hand side individually.  According to \cite[Lemmas 4.7 and 5.6, Thm.\ 6]{dfg}, there is for $\Re\lambda >-1$ an isomorphism
\begin{equation}\label{eq:23989089098105}
\tu{res}^{\mathrm{1st}}_{\nabla_X,S^2_0(E^\ast_-)}(\lambda)\cong 
\{\omega\in \Gamma^\infty(S^2_0(T^\ast \M)):\Delta_B\omega=-\lambda(n+\lambda)+2, ~\mathrm{div}\,\omega=0\},
\end{equation}
where $\Delta_B$ is the Bochner Laplacian associated to the connection $\nabla$. The eigenvalue $-\lambda(n+\lambda)+2$ appearing here is a real number iff $\Im \lambda=0$ or $\Re \lambda=-\frac{n}{2}$, so for $\Re\lambda >-\frac{1}{2}$  only numbers $\lambda\in(-1/2,\infty)$ remain as  possible candidates for a non-zero resonance space \eqref{eq:23989089098105}. 
In addition, a Weitzenböck type formula (see \cite[Lemma 6.1]{dfg}) says that the spectrum of $\Delta_B$ acting on $\Gamma^\infty(S^2_0(T^\ast \M))$ is bounded from below by $n+1$ 
 which is strictly larger than $-\lambda(n+\lambda)+2$ for $n\geq 2$ and $\lambda\in(-1/2,\infty)$. 
Consequently, for such $n$ and $\lambda$ the right hand side of \eqref{eq:23989089098105} is the zero space and it follows that $\Res^{\mathrm{1st}}_{\nabla_X,S^2_0(E^\ast_-)}(\lambda)=\{0\}$ because every Jordan block would contain at least one resonant state.
Turning to the second summand in \eqref{eq:3693sf9sgs0}, 
we apply once more Proposition~\ref{prop:boundary_distribution}
and Theorem \ref{thm:gaillard} and obtain for 
$n\neq 2$ an isomorphism 
\bq
\mathrm{res}^{\mathrm{1st}}_{\nabla_X,\Lambda^2E^\ast_-}(\lambda)\cong\{\omega \in \Gamma^\infty(\Lambda^2(T^\ast \M)):\Delta_H\omega=-(\lambda+2)(n+\lambda-2), ~\delta \omega=0\}.\label{eq:1stLambda}
\eq
For $\Re\lambda>-1$ and $n\geq 3$, the eigenvalue appearing here is either imaginary or negative, so the right hand side of \eqref{eq:1stLambda} is the zero space (because $\Delta_H$ is positive) and $\mathrm{res}^{\mathrm{1st}}_{\nabla_X,\Lambda^2E^\ast_-}(\lambda)=\{0\}$, $\Res^{\mathrm{1st}}_{\nabla_X,\Lambda^2E^\ast_-}(\lambda)=0$. 

When $n=2$ we have $\Lambda^2E^\ast_-=\R\Omega_{E_-}$. We can thus treat the second summand in \eqref{eq:3693sf9sgs0}  for $n=2$ and the third summand in \eqref{eq:3693sf9sgs0} for arbitrary $n$ in the same way: 
As $\nabla_X^J (\tilde c\Omega_{E_-})=(X^J\tilde c)\,\Omega_{E_-}$ and $\nabla_X^J (c\,\tilde{\mathscr g}|_{E_-\times E_-})=(X^Jc)\,{\mathscr g}|_{E_-\times E_-}$ for each $J\in \N$, we see that the distributions $c,\tilde c$ have to be generalized scalar resonant states of a resonance $\lambda$. In the scalar case we can however apply Liverani's result on the spectral gap for contact Anosov flows \cite{Liv04} to see that zero is the unique leading resonance, with (generalized) resonant states the locally constant functions, and there is a spectral gap $\delta>0$,  so the proof is finished.
\end{proof}

\subsection{Gaillard's Poisson transform}\label{subsec:gaillard}
In his article \cite{gaillard1} Gaillard considers the vector-valued Poisson transform to which we refer in Theorem~\ref{thm:gaillard} in the special case of $\Gamma$-invariant elements. His notation and 
conventions are however quite different from ours. In the following we will translate his results into the form stated in Theorem~\ref{thm:gaillard}. 

Gaillard proves in \cite[Therems 2', 3']{gaillard1} that \emph{slowly growing} co-closed $p$-forms on $\mathbb H^{n+1}$ in appropriate eigenspaces of the Hodge Laplacian on $\mathbb H^{n+1}$ are the Poisson transforms of  $p$-currents on $K/M$. When considering only $p$-forms on $\mathbb H^{n+1}$ that are $\Gamma$-invariant with respect to the action of $\Gamma$ by pullbacks, and which we identify with $p$-forms on the compact quotient $\M=\gam\mathbb H^{n+1}$ in Theorem \ref{thm:gaillard}, the slow growth condition becomes redundant. The remaining task is to relate Gaillard's pullback $G$-actions on $p$-currents to our principal series representations of $G$ on distributional sections. 

We will denote the space of $p$-currents on $K/M$ by
$\mathcal D_p'(K/M):=(\Omega^{n-p}(K/M))'$, and we have the canonical 
dense embedding $\Omega^p(K/M)\hooklongrightarrow \mathcal D_p'(K/M)$. As $G$ acts by diffeomorphisms on $K/M$ the pullback action on $\mathcal D'_p(K/M)$ provides a 
$G$-representation.
\begin{lem}\label{lem:compat}
The pullback action of $G$ on the space $\mathcal D'_p(K/M)$ of  $p$-currents is equivalent to the principal series representation $\pi^{p-n/2}_{\tau_p}$ on $\D'(K/M,\mathcal V^{\mathcal B}_{\tau_p})$. 
\end{lem}
\begin{proof}
Denote by $\m^{\perp_\k}\subset \k$ the orthogonal complement of $\m$ in $\k$. Then $M$ acts via the adjoint action on $\m^{\perp_\k}$. Recall from Section \ref{subsec:Lie} that  
$\m^{\perp_\k} \cong \R^n$ and $\tu{Ad}(M)|_{\m^{\perp_\k}}$ is 
nothing but the standard action of $\SO(n)$ on $\R^n$.
In the following, we shall write simply $\mathrm{Ad}(M)$ instead of $\mathrm{Ad}(M)|_{\m^{\perp_\k}}$. Note that there is a canonical identification
\begin{equation}\label{eq:ident_tangen_KM} 
K\times_{\mathrm{Ad}(M)}\m^{\perp_\k}\cong T(K/M)~\tu{ by }~[k,Y]\mapsto \frac{d}{dt}\Big|_{t=0} ke^{tY}M. 
\end{equation}
Let $g\in G$ and $\alpha_g: kM \mapsto k_-(gk)M$ be the 
diffeomorphism on $K/M$ given by the left-$G$-action, then the derivative $d\alpha_g$ acts on $T(K/M)$. In order to prove our lemma we have to determine
how $d\alpha_g$ acts on $K\times_{\mathrm{Ad}(M)}\m^{\perp_\k}$ under the
identification \eqref{eq:ident_tangen_KM}. We have for $[k,Y] \in T(K/M)$
\begin{align}\begin{split}
d\alpha_g\left( [k,Y]\right) = \frac{d}{dt}\Big|_{t=0} k^-(gke^{Yt}) M&\cong \left[k^-(gk),\frac{d}{dt}\Big|_{t=0}k^-(gk)^{-1}k^-(gk\exp(tY))\right]\\
&=\left[k^-(gk),\frac{d}{dt}\Big|_{t=0}k^-\big(a^-(gk)n^-(gk)\exp(tY)n^-(gk)^{-1}a^-(gk)^{-1}\big)\right]\\
&=\left[k^-(gk),\pr_\k^-\mathrm{Ad}(a^-(gk)n^-(gk))(Y)\right],\label{eq:DLkL}\end{split}
\end{align}
where
\begin{equation}
\pr_\k^-={d} k^-|_e: \g \to \g=\k\oplus\aL\oplus \nL^- \label{eq:Dk}
\end{equation} 
is the projection onto $\k$ defined by the opposite Iwasawa decomposition of $\g$. 

We can now proceed by studying for fixed $g\in G$, $k\in K$, $Y\in \m^{\perp_\k}$ the element
\begin{equation}
 \pr_\k^-\mathrm{Ad}(a^-(gk)n^-(gk))(Y) \in \m^{\perp_\k}.\label{eq:el}
\end{equation}
By the orthogonal Bruhat decomposition $\g=\m\oplus \aL\oplus \nL^+\oplus\nL^-$ and the fact that $\aL$ lies in the orthogonal complement of $\k$ in $\g$, we have $\m^{\perp_\k}\subset \nL^+\oplus\nL^-$, so we can write $Y=Y^++Y^-$ with $Y^\pm \in \nL^\pm$ and $\theta Y^\pm = Y^\mp$. 
The space $\nL^\pm$ is $\Ad(AN^\pm)$-invariant. Consequently $\mathrm{Ad}(a^-(gk)n^-(gk))(Y^-)\in \nL^-$, so $\pr_\k^-\mathrm{Ad}(a^-(gk)n^-(gk))(Y^-)=0$ 
by the opposite Iwasawa decomposition. This shows that only $Y^+$ contributes to \eqref{eq:el}. 
Let us write $n^-(gk)=\exp(N)$ with $N\in \nL^-$. Then we get
\[
\mathrm{Ad}(n^-(gk))(Y^+)=e^{\mathrm{ad}(N)}(Y^+)=Y^+  +\underbrace{[N,Y^+]}_{\in \g_{0}=\m\oplus \aL}+ 
 \frac{1}{2} \underbrace{[N,[N,Y^+]]}_{\in \nL_-}.
\]
Here we use that $\g= \g_0 \oplus \nL^+\oplus \nL^-$ is the root-space 
decomposition of $\g=\mathfrak{so}(n+1,1)$ and consequently
\[
 \nL^+\overset{\ad(N)}{\longrightarrow}\g_0\overset{\ad(N)}{\longrightarrow}\nL^-\overset{\ad(N)}{\longrightarrow}0.
\]
Furthermore, the map $\mathrm{Ad}(a^-(gk))$ acts on $\nL^\pm$ by scalar
multiplication with $e^{\pm H^-(gk)}$ and leaves $\g_0=\m\oplus \aL$ invariant. 
The opposite Iwasawa projection  $\pr_\k^-$ maps $\nL^-$ to $0$ and  the space $\g_{0}$ onto $\m$. However, the Lie algebra element considered in \eqref{eq:el} is by construction in $\m^{\perp_\k}$. We therefore arrive at
\[
\pr_\k^-\mathrm{Ad}(a^-(gk)n^-(gk))(Y)=\pr_\k^-\left(e^{ H^-(gk)}Y^+\right).
\]
Writing 
\[
Y^+ =\underbrace{Y^+ + \theta Y^+}_{\in \k}- \underbrace{\theta Y^+}_{\in \nL^-}
~~\tu{ reveals }~~
\pr_\k^-\mathrm{Ad}(a^-(gk)n^-(gk))(Y)
=e^{H^-(gk)}Y.
\]
In summary, we have proved that
\begin{equation}\label{eq:differential_on_KM}
d\alpha_g ([k,Y])=\big[k^-(gk), e^{H^-(gk)}Y\big].
\end{equation}
Finally, note that $T(K/M)\cong K\times_{\mathrm{Ad}(M)}\m^{\perp_\k}$ induces for each $p\in \{1,2,\ldots\}$ an isomorphism $\Lambda^p T^\ast(K/M)\cong K\times_{\Lambda^p\mathrm{Ad}^\ast(M)}\Lambda^p(\m^{\perp_\k})^\ast$. Under that isomorphism,
a $p$-form $s\in \Gamma^\infty(\Lambda^p T^\ast(K/M))$ corresponds to a section $\hat s\in \Gamma^\infty(K\times_{\Lambda^p\mathrm{Ad}^\ast(M)}\Lambda^p(\m^{\perp_\k})^\ast)$, and by our above computations the pullback action $gs\equiv (g^{-1})^\ast s$ of an element $g\in G$ on $s$ corresponds to the following action on $\hat s$: 
\begin{align}\begin{split}
\overline{(g \hat s)}(k)(X_1,\ldots,X_p)&=\overline{\hat s}(k^-(g^{-1}k))(e^{H^-(g^{-1}k)}X_1,\ldots,e^{H^-(g^{-1}k)}X_p)\\
&=e^{pH^-(g^{-1}k)}\overline{\hat s}(k^-(g^{-1}k))(X_1,\ldots,X_p)\qquad\forall\; X_1,\ldots,X_p\in \nL^\pm,\;k\in K.\label{eq:58036809182}\end{split}
\end{align}
Recalling the definition \eqref{eq:compactp} of the principal series representations, and taking into account that the pullback action of $G$ on $p$-currents as well as the principal series representations of $G$ on  distributional sections of $\Lambda^pT^\ast(K/M)$ are the continuous extensions of the respective actions on smooth $p$-forms, the proof is complete.
\end{proof}
For the definition of his Poisson transform Gaillard generalizes his setting to currents with values in complex line bundles $D^s\to K/M$ parametrized by a complex number $s\in\C$. Let us recall their construction \cite[Section 2.2]{gaillard1}: 
It is based on a $G$-invariant function\footnote{Here $G$ acts on all three factors in the domain $G/K\times K/M\times G/K$ by left multiplication.}
\begin{equation}Q:G/K\times K/M\times G/K\to \C\setminus\{0\},\qquad 
Q(gK,kM,eK)=\Vert D(V^{-1}_{gK}\circ V_{eK})|_{kM} \Vert,\label{eq:Q}
\end{equation}
where Gaillard's ``application visuelle'' $V_{gK}: S^\ast_{gK}(G/K)\to K/M$, $gK\in G/K$, is defined by
\bqn
V_{gK}: \{\tilde gM:\tilde gK=gK\}=S^\ast_{gK}(G/K)\to K/M,\qquad 
\tilde gM \mapsto k^-(\tilde g) M.
\eqn
A straightforward calculation similar to the proof of Lemma~\ref{lem:compat} shows that
\begin{equation}
Q(gK,kM,eK)=e^{H^-(g^{-1}k)},
\label{eq:Q2}
\end{equation}
which gives us by the $G$-invariance of $Q$ for a general element $(\tilde gK,kM,gK)\in G/K\times K/M\times G/K$:
\begin{align}\begin{split}
Q(\tilde gK,kM,gK)&=Q(g(g^{-1}\tilde gK,k^-(g^{-1}k)M,eK))\\
&=Q(g^{-1}\tilde gK,k^-(g^{-1}k)M,eK)\\
&=e^{H^-(\tilde g^{-1}g k^-(g^{-1}k))}.
\label{eq:Q3}
\end{split}
\end{align}
With these preparations, let us now turn to Gaillard's definition of the line bundle $D^s$ over $K/M$: Introduce an equivalence relation $\sim_s$ on $G/K\times K/M\times \C$ by
\begin{equation*}
(gK,kM,z)\sim_s (\tilde gK,\tilde kM,\tilde z)
\iff kM=\tilde kM,\;\tilde z =Q(\tilde gK,kM,gK)^{-s} z=e^{-sH^-(\tilde g^{-1}g k^-(g^{-1}k))}z,
\end{equation*}
and declare $D^s:=G/K\times K/M\times \C/\sim_s$ with bundle projection  $[gK,kM,z]\mapsto kM$. The bundle is a homogeneous $G$-bundle by defining the $G$ action as 
\[
g'[gK,kM,z]:=[g'gK,g'(kM),z]=[g'gK,k^-(g'k)M,z].
\]
The stabilizer subgroup of $eM\in K/M$ with respect to the left-$G$-action on $K/M$ is $MAN^-$ and the action of the stabilizer group on the fiber of $D^s$ over $eM$ is 
\[
 [manK , eM, z] = [eK, eM, e^{-sH(mank^-(n^{-1}a^{-1}m^{-1}))}z]=  [eK, eM,  e^{-s\log(a)}z].
\]
If we define the $MAN^-$-representation $\sigma_s$ by $man\mapsto e^{-s\log(a)} \in \C$ then we can identify $D^s$ with the associated line bundle
$G \times_{\sigma_s}\C \to G/(MAN^-)\cong K/M$. Thus the $G$-action on sections of this homogenous bundle is equivalent to the principle series representation
$\pi^{s}_{\mathbb 1}$, where $\mathbb 1$ denotes the trivial $M$-representation on $\C$. By Lemma~\ref{lem:compat} we know that the pullback action on $p$-currents is equivalent to $\pi^{p-n/2}_{\tau_p}$, so the action of $G$ on $D^s$-valued currents is
equivalent to $\pi^{p-n/2}_{\tau_p}\otimes \pi^{s}_{\mathbb{1}}$
which is equivalent to $\pi^{p+s-n/2}_{\tau_p}$.

\section{Non-constant curvature perturbations}\label{sec:nonconst}
We now address the question how the equality 
$m_{\mathcal L_X, X^\perp}(0)=b_1(\mathcal M)$ for
constant negative curvature manifolds behaves under 
perturbations of the Riemannian metric and also under more general perturbations of the vector field $X$ that do not (only) result from metric perturbations. Throughout 
this section, let $\mathcal M$ be a closed orientable  manifold  admitting a hyperbolic metric and $\Gamma^\infty(\mathrm{S}^2(T^*\M))$ the space of symmetric two-tensors endowed with the Fr\'echet topology. 
Let $\mathscr R_{\mathcal M,<0}\subset \Gamma^\infty(\mathrm{S}^2(T^*\M))$ be the open 
subset of Riemannian metrics of negative sectional 
curvature. 
For any Riemannian metric $\mathscr g$ on $\M$, we write 
$X_{\mathscr g}\in \Gamma^\infty(T(S_{\mathscr g}^*\mathcal M))$ for the 
geodesic vector field on the unit sphere bundle $S_{\mathscr g}^*\mathcal M$ with respect to $\mathscr g$. In order to study perturbations of the vector field $X_{\mathscr g}$, we consider $S_{\mathscr g}^*\mathcal M$ as a Riemannian manifold equipped with the metric $\tilde {\mathscr g}$ induced by the Sasaki metric on $T^\ast \mathcal M$ with respect to $\mathscr g$  and define the $\mathrm{C}^1$-norm on $\Gamma^\infty(T(S_{\mathscr g}^*\mathcal M))$ by
\bq
\norm{Y}_{\mathrm{C}^1}:=\sup_{\xi \in S_{\mathscr g}^*\mathcal M}\big(\norm{Y(\xi)}_{\tilde {\mathscr g}} +\norm{(\nabla^{\tilde {\mathscr g}}Y)(\xi)}_{\tilde {\mathscr g}}\big),\qquad Y\in \Gamma^\infty(T(S_{\mathscr g}^*\mathcal M)), \label{eq:C1norm}
\eq
where $\nabla^{\tilde {\mathscr g}}:\Gamma^\infty(T(S_{\mathscr g}^*\mathcal M))\to \Gamma^\infty(T(S_{\mathscr g}^*\mathcal M)\otimes T^\ast(S_{\mathscr g}^*\mathcal M))$ is the Levi-Civita connection with respect to $\tilde {\mathscr g}$ and we denoted the metric obtained by extending $\tilde {\mathscr g}$ to the tensor bundle $T(S_{\mathscr g}^*\mathcal M)\otimes T^\ast(S_{\mathscr g}^*\mathcal M)$ also by $\tilde {\mathscr g}$.

With this notation at hand, we can  prepare the proof of our main Theorem \ref{thm:main2} which will be given on page \pageref{proof:perturb}. As already indicated in the introduction, we essentially reduce the proof to two steps: Lemma \ref{lem:perturbation} will provide a local upper bound for the multiplicity of an arbitrary resonance, while Lemma \ref{lem:lower_bound} will provide global lower bounds for the resonance zero. Finally, Lemma \ref{lem:perpall} relates the multiplicities of the resonance zero on general and on perpendicular one-forms.
\begin{lem}\label{lem:perturbation}
 For each $\lambda\in\C$ and each $\mathscr g_0\in \mathscr R_{\mathcal M,<0}$ there is an open set $U\subset \mathscr R_{\mathcal M,<0}$ containing $\mathscr g_0$ and a constant $\delta>0$ such that for all $\mathscr g\in U$ and all $Y_{\mathscr g}\in \Gamma^\infty(T(S_{\mathscr g}^*\mathcal M))$ with $\Vert Y_{\mathscr g}-X_{\mathscr g}\Vert_{\mathrm{C}^1}<\delta$ one has $$m_{\mathcal L_{Y_{\mathscr g}},T^\ast(S_{\mathscr g}^*\mathcal M)}(\lambda)\leq m_{\mathcal L_{X_{\mathscr g_0}}, T^\ast(S_{\mathscr g_0}^*\mathcal M)}(\lambda).$$
\end{lem}
\begin{proof} Fix some reference metric $\mathscr g_0\in \mathscr R_{\mathcal M,<0}$ for the rest of the proof and let $\mathscr g$ be some arbitrary Riemannian metric on $\M$. Then the diffeomorphism
$
\phi_{\mathscr g}: S_{\mathscr g}^*\mathcal M\to S_{\mathscr g_0}^*\mathcal M$, $\xi \mapsto \norm{\xi}_{\mathscr g_0}^{-1}\xi
$, 
fulfills
\bq
\Vert d\phi_{\mathscr g}v\Vert_{\tilde {\mathscr g}_0}=\norm{\xi}_{\mathscr g_0}^{-1}\norm{v}_{\tilde {\mathscr g}}\qquad \forall\; v\in T_{\xi}(S_{\mathscr g}^*\mathcal M),\;\xi \in S_{\mathscr g}^*\mathcal M .\label{eq:scalingdphi}
\eq
For a vector field $Y_{\mathscr g}\in \Gamma^\infty(T(S_{\mathscr g}^*\mathcal M))$, consider its pushforward $Y_{\ast\mathscr g}:=(\phi_{\mathscr g})_\ast Y_{\mathscr g}\in \Gamma^\infty(T(S_{\mathscr g_0}^*\mathcal M))$. By the naturality of the Lie derivative with respect to pullbacks, the following diagram commutes:
\begin{center}\begin{tikzpicture}[node distance=1cm and 2cm, auto]
\node (A) {$\D'(S_{\mathscr g}^*\mathcal M,T^\ast(S_{\mathscr g}^*\mathcal M) )$}; 
\node (B) [right= of A] {$\D'(S_{\mathscr g}^*\mathcal M,T^\ast(S_{\mathscr g}^*\mathcal M) )$};
\node (C) [below= of A] {$\D'(S_{\mathscr g_0}^*\mathcal M,T^\ast (S_{\mathscr g_0}^*\mathcal M))$};
\node (D) [below= of B] {$\D'(S_{\mathscr g_0}^*\mathcal M,T^\ast(S_{\mathscr g_0}^*\mathcal M) )$};
\draw[->] (A) to node {$\mathcal L_{Y_{\mathscr g}}$} (B); 
\draw[->] (C) to node {$\mathcal L_{Y_{\ast\mathscr g}}$} (D); 
\draw[<-] (A) to node {$\phi_{\mathscr g}^\ast$} (C);
\draw[->] (B) to node {$(\phi_{\mathscr g}^{-1})^\ast$} (D);  
\end{tikzpicture}\end{center}
By comparing the pushforward connection $(\phi_{\mathscr g})_\ast\nabla^{\tilde {\mathscr g}}$ with $\nabla^{\tilde {\mathscr g_0}}$ using the Koszul formula, one checks that
\bqn
\Vert Y_{\ast\mathscr g}-Y'_{\ast\mathscr g}\Vert_{\mathrm{C}^1}\leq C_{\mathscr g}\Vert Y_{\mathscr g}-Y'_{\mathscr g}\Vert_{\mathrm{C}^1}\qquad \forall\; Y_{\mathscr g},Y'_{\mathscr g}\in \Gamma^\infty(T(S_{\mathscr g}^*\mathcal M))
\eqn
with a constant $C_{\mathscr g}>0$ that depends continuously on $\mathscr g$ with respect to the Fr\'echet topology on $\Gamma^\infty(\mathrm{S}^2(T^*\M))$. Furthermore, the geodesic vector fields $X_{\mathscr g}$ and  $X_{\mathscr g_0}$ fulfill 
\bqn
\Vert X_{\ast\mathscr g}-X_{\mathscr g_0}\Vert_{\mathrm{C}^1}\to 0\quad \text{as }\mathscr g\to \mathscr g_0\text{ in } \Gamma^\infty(\mathrm{S}^2(T^*\M)).
\eqn
Thus, for every $\eps>0$ we can find an open set $U\subset \mathscr R_{\mathcal M,<0}$ containing $\mathscr g_0$ and a $\delta>0$ such that 
\bq
\Vert Y_{\ast\mathscr g}-X_{\mathscr g_0}\Vert_{\mathrm{C}^1}< \eps\qquad \forall\; \mathscr g\in U,\; Y_{\mathscr g}\in \Gamma^\infty(T(S_{\mathscr g}^*\mathcal M)): \Vert Y_{\mathscr g}-X_{\mathscr g}\Vert_{\mathrm{C}^1}< \delta.\label{eq:continuityY}
\eq
Choosing $\eps$ small enough, the structural stability (see \cite[Thm.\ A]{kato73}) of the Anosov property of vector fields on the $\mathscr g$-independent manifold $S_{\mathscr g_0}^*\mathcal M$ allows us to assume from now on that $Y_{\ast\mathscr g}$ is Anosov for all $\mathscr g\in U$. Then also $Y_{\mathscr g}$ is Anosov for all $\mathscr g\in U$. Indeed, the Anosov splitting of $Y_{\mathscr g}$ is obtained by applying $d\phi_{\mathscr g}^{-1}$ to the Anosov splitting of $Y_{\ast\mathscr g}$. 
 In view of the commutative diagram above one has
\bq
  m_{\mathcal L_{Y_{\mathscr g}},T^\ast(S_{\mathscr g}^*\mathcal M)}(\lambda)=m_{\mathcal L_{Y_{\ast\mathscr g}},T^\ast(S_{\mathscr g_0}^*\mathcal M) }(\lambda)\qquad \forall\; \lambda \in \C.\label{eq:multiplagree}
 \eq 
Given some $\lambda \in \C$ we now apply the perturbation result \cite{Bon18}, which says that on every closed manifold the resonances of all Anosov vector fields $Y$ that are $\mathrm{C}^1$-close to a given Anosov vector field $Y_0$ can be defined as eigenvalues in certain Hilbert spaces that depend only on $Y_0$ and not on $Y$, so that the change of the multiplicity of $\lambda$ in this fixed Hilbert space can be measured as $Y$ varies near $Y_0$. The results of \cite{Bon18} generalize  easily to a vector-valued situation (for vector bundles that do not vary with the vector field $Y$) by replacing the scalar  quantization map in \cite[Eq.\ (2)]{Bon18} by a vector-valued quantization map. The correspondingly generalized  \cite[Cor.\ 2]{Bon18} then implies that there is a  $c >0$ such that all $Y\in \Gamma^\infty(T(S_{\mathscr g_0}^*\mathcal M))$ with $\Vert Y-X_{\mathscr g_0} \Vert_{\mathrm{C}^1}< c$ fulfill \bq
m_{\mathcal L_{Y},T^\ast(S_{\mathscr g_0}^*\mathcal M)}(\lambda)\leq m_{\mathcal L_{X_{\mathscr g_0}},T^\ast(S_{\mathscr g_0}^*\mathcal M)}(\lambda).\label{eq:ineqBonth}
\eq
Choosing $\eps<c$ in \eqref{eq:continuityY}, we can put $Y=Y_{\ast \mathscr g}$ in \eqref{eq:ineqBonth} for each $\mathscr g\in U$, and by \eqref{eq:multiplagree} the proof is finished.
\end{proof}
A second ingredient to Theorem~\ref{thm:main2} is a very general lower bound on the  multiplicity of the resonance zero:
\begin{lem}~\label{lem:lower_bound}
 For some Riemannian metric $\mathscr g$ on $\M$, let  $Y_{\mathscr g}\in \Gamma^\infty(T(S_{\mathscr g}^*\mathcal M))$ be an Anosov vector field. Then $$m_{\mathcal L_{Y_{\mathscr g}},T^*(S_{\mathscr g}^*\mathcal M)}(0)\geq b_1(S^*_{\mathscr g}\mathcal M),$$ 
 and if there is a one-form $\alpha_{\mathscr g}$ on $S_{\mathscr g}^*\mathcal M$  with  $\mathcal L_{Y_{\mathscr g}}\alpha_{\mathscr g}=0$ and $d\alpha_{\mathscr g}\neq 0$, then  $$m_{\mathcal L_{Y_{\mathscr g}},T^*(S_{\mathscr g}^*\mathcal M)}(0)\geq b_1(S^*_{\mathscr g}\mathcal M) +b_0(S^*_{\mathscr g}\mathcal M).$$
\end{lem}
\begin{rem} Lemma \ref{lem:lower_bound} remains true, with the same proof, if $S_{\mathscr g}^*\mathcal M$ is replaced by an arbitrary closed oriented manifold $\mathcal N$ and $Y_{\mathscr g}$ by an Anosov vector field $Y\in \Gamma^\infty(T\mathcal N)$. 
\end{rem}
\begin{proof}[Proof of Lemma \ref{lem:lower_bound}]
Fix some Riemannian metric $\mathscr g$ on $\M$. Dang-Rivi\`ere \cite{DR17b} proved that for every Anosov vector field $Y_{\mathscr g}\in \Gamma^\infty(T(S_{\mathscr g}^*\mathcal M))$ 
 \[
  0\to\tu{Res}_{\mathcal L_{Y_{\mathscr g}},\Lambda^0(T^*(S_{\mathscr g}^*\M))}(0) \overset{d}{\to}\tu{Res}_{\mathcal L_{Y_{\mathscr g}},\Lambda^1(T^*(S_{\mathscr g}^*\M))}(0)
  \overset{d}{\to}\cdots\overset{d}{\to} \tu{Res}_{\mathcal L_{Y_{\mathscr g}},\Lambda^{2\dim \M-1}(T^*(S_{\mathscr g}^*\M))}(0)\to 0
 \]
forms a finite-dimensional complex whose cohomology is isomorphic to the de Rham cohomology of $S_{\mathscr g}^*\mathcal M$. This implies $\dim (\tu{Res}_{\mathcal L_{Y_{\mathscr g}},T^*(S_{\mathscr g}^*\M)}(0) \cap \tu{ker}(d)) \geq b_1(S^*_{\mathscr g}\mathcal M)$, proving the first inequality. Now suppose that there is a one-form $\alpha_{\mathscr g}$ with  $\mathcal L_{Y_{\mathscr g}}\alpha_{\mathscr g}=0$ and $d\alpha_{\mathscr g}\neq 0$. By the wave front characterization of resonant states \eqref{eq:WF_characterization} we then know that $c\alpha_{\mathscr g} \in \tu{Res}_{\mathcal L_{Y_{\mathscr g}},T^*(S_{\mathscr g}^*\M)}(0)$ for each locally constant function $c$ on $S_{\mathscr g}^*\M$ (thus each element in the $0$-th de Rham cohomology). Since $d(c\alpha_{\mathscr g}) \neq 0$ if $c\neq 0$, the second inequality follows.
\end{proof}

We can now prove Theorem~\ref{thm:main2}:
\begin{proof}[Proof of Theorem \ref{thm:main2}]\label{proof:perturb}
Assume that $\dim \M\neq 3$ and let $\mathscr g_0\in \mathscr R_{\mathcal M,<0}$ be a metric of constant negative curvature. Then we can apply Proposition~\ref{prop:constant_mult_perp} and \eqref{eq:allforms} with $\mathscr g=\mathscr g_0$, $Y_{\mathscr g_0}=X_{\mathscr g_0}$, and $\alpha_{\mathscr g_0}$ the canonical contact form on $S^*_{\mathscr g_0}\mathcal M$ to get 
\bq
 m_{\mathcal L_{X_{\mathscr g_0}},T^\ast(S^*_{\mathscr g_0}\mathcal M)}(0) =m_{\mathcal L_{X_{\mathscr g_0}},X_{\mathscr g_0}^\perp}(0)+ b_0(\M)=b_1(\mathcal M)+b_0(\M).\label{eq:mb1}
\eq
Now, if $\mathscr g$ is any Riemannian metric on $\M$, then by \cite[(4.1)]{CS50} one has $b_1(\mathcal M)=b_1(S^*_{\mathscr g}\mathcal M)$, and we also have $b_0(\mathcal M)=b_0(S^*_{\mathscr g}\mathcal M)$ because $\dim \M\geq 2$ (otherwise $\M$ would not admit metrics of negative sectional curvature). Thus, it suffices to apply the local upper bound from Lemma \ref{lem:perturbation} for $\lambda=0$ and the global lower bounds from Lemma  \ref{lem:lower_bound} to finish  the proof of Theorem~\ref{thm:main2}.
\end{proof}
Finally, in order to get a statement involving resonance multiplicities on the bundle $Y_{\mathscr g}^\perp$, one can use the following basic result:
\begin{lem}\label{lem:perpall}For some Riemannian metric $\mathscr g$ on $\M$, let  $Y_{\mathscr g}\in \Gamma^\infty(T(S_{\mathscr g}^*\mathcal M))$ be an Anosov vector field. 

If there is a one-form $\alpha_{\mathscr g}$ on $S_{\mathscr g}^*\mathcal M$ with $\iota_{Y_{\mathscr g}}\alpha_{\mathscr g}=1$, $\iota_{Y_{\mathscr g}}d\alpha_{\mathscr g}=0$, and $T^\ast(S^\ast_{\mathscr g}\M)=\R\alpha_{\mathscr g} \oplus Y_{\mathscr g}^\perp$, then
\bq
m_{\mathcal L_{Y_{\mathscr g}},T^*(S_{\mathscr g}^*\mathcal M)}(0)=m_{\mathcal L_{Y_{\mathscr g}}, Y_{\mathscr g}^\perp}(0)+ b_0(\mathcal M).\label{eq:allforms}
\eq
\end{lem}
\begin{proof}
As $\iota_{Y_{\mathscr g}} \alpha_{\mathscr g} = 1$ and $T^\ast(S^\ast_{\mathscr g}\M)=\R\alpha_{\mathscr g} \oplus Y_{\mathscr g}^\perp$, we can uniquely decompose every $u\in \tu{Res}_{\mathcal L_{Y_{\mathscr g}}, T^*(S_{\mathscr g}^*\M)}(0)$ into $u=u_\perp + (\iota_{Y_{\mathscr g}} u)\alpha_{\mathscr g}$ where $\iota_{Y_{\mathscr g}} u_\perp=0$ and thus $u_\perp$
 is a distributional section of $Y_{\mathscr g}^\perp$. We have for some $J\in \N$
 \begin{equation}
  0=\mathcal L_{Y_{\mathscr g}}^J u = \mathcal L_{Y_{\mathscr g}}^J u_\perp + (Y_{\mathscr g}^J (\iota_{Y_{\mathscr g}} u))\alpha_{\mathscr g},\label{eq:18129059812}
 \end{equation}
 since $\iota_{Y_{\mathscr g}} \alpha_{\mathscr g} = 1$, $\iota_{Y_{\mathscr g}}d\alpha_{\mathscr g}=0$ implies $\mathcal L_{Y_{\mathscr g}}\alpha_{\mathscr g} =0$.  
Using Cartan's magic formula one checks $\iota_{Y_{\mathscr g}} \mathcal L_{Y_{\mathscr g}}^J u_\perp= 0$, so the wave front characterization of resonant states \eqref{eq:WF_characterization}  implies $u_\perp\in \Res_{\mathcal L_{Y_{\mathscr g}},Y_{\mathscr g}^\perp}(0)$ and $\iota_{Y_{\mathscr g}} u \in \Res_{Y_{\mathscr g}}(0)$. The latter space is of dimension $b_0(\mathcal M)$ as it consists of the locally constant functions (cf.\ the end of the proof of Lemma 
\ref{lem:sym_tensor_resonances}), so we get \eqref{eq:allforms}.
\end{proof}

\providecommand{\bysame}{\leavevmode\hbox to3em{\hrulefill}\thinspace}
\providecommand{\MR}{\relax\ifhmode\unskip\space\fi MR }
\providecommand{\MRhref}[2]{%
  \href{http://www.ams.org/mathscinet-getitem?mr=#1}{#2}
}
\providecommand{\href}[2]{#2}

\end{document}